\documentclass{amsart}

\usepackage{xypic,hyperref,todonotes,comment}
\usepackage{graphics,amssymb,amsfonts,amscd,verbatim}
\usepackage{pagecolor,lipsum}
\usepackage[all]{xy}
\usepackage{latexsym}
\usepackage{mathrsfs,mathtools,stmaryrd,tikz-cd}
\xyoption{curve}
\usepackage{hyperref}
\hypersetup{ 
colorlinks,
citecolor=blue,
filecolor=black,
linkcolor=black,
urlcolor=green
}

\newtheorem{theorem}{Theorem}[section]
\newtheorem{lemma}[theorem]{Lemma}
\newtheorem{proposition}[theorem]{Proposition}
\newtheorem{corollary}[theorem]{Corollary}
\theoremstyle{definition}

\newtheorem{example}[theorem]{Example}
\newtheorem{question}[theorem]{Question}

\newtheorem{remark}[theorem]{Remark}

\numberwithin{equation}{section}

\DeclareMathOperator{\id}{id}
\DeclareMathOperator{\Hom}{Hom}
\DeclareMathOperator{\End}{End}

\DeclareMathOperator{\Fun}{Fun}

\DeclareMathOperator{\Rep}{Rep}

\DeclareMathOperator{\Sym}{Sym}

\newcommand{\ot}{\otimes}
\newcommand{\Ver}{\mathscr{V}\!er}
\renewcommand{\1}{\mathbf{1}}
\renewcommand{\O}{\mathscr{O}}
\newcommand{\C}{\mathscr{C}}
\newcommand{\Z}{\mathscr{Z}}

\title[On the Drinfeld center of the Verlinde category $\Ver_p$]{On the Drinfeld center of the Verlinde category}
\date{\today}

\author{Shlomo Gelaki}
\address{Department of Mathematics, Iowa State University, Ames, IA 50011, USA} 
\email{gelaki@iastate.edu}
\author{Victor Ostrik}
\address{Department of Mathematics, University of Oregon, Eugene, OR 97403, USA}
\email{vostrik@uoregon.edu}

\begin{document}

\maketitle

\tableofcontents

\begin{abstract}
We provide some information about the Drinfeld center $\Z(\Ver_p)$ of the Verlinde category $\Ver_p$.
We compute explicitly the graded Cartan matrix, bound quiver and cohomology of $\Z(\Ver_p^+)$, and prove that the category $\mathscr{Z}(\Ver_p^+)$ is wild for every $p\ge 7$. In the special case $p=5$, we show that $\Z(\Ver_5^+)$ has exactly $10$ non-isomorphic indecomposable objects, classify them and describe the Green ring of $\Z(\Ver_5^+)$, and compute the semisimplification of $\Z(\Ver_5^+)$.
\end{abstract}

\section{Introduction}\label{sec:Introduction}

The goal of this paper is to collect some results about the Drinfeld center of the Verlinde category $\Ver_p$.
Recall that the Verlinde category (first appearing in the work of Gelfand and Kazhdan \cite{GK} and
Georgiev and Mathieu \cite{GM}) is a symmetric fusion category over an algebraically closed field $\mathbf{k}$ of positive characteristic $p>0$;
this category plays an important role in the general theory of symmetric tensor categories, for instance 
it is proved in \cite{CEO} that any semisimple symmetric tensor category of moderate growth admits
a symmetric tensor functor to $\Ver_p$. We recall (see 
\S\ref{sec:The category Verp}) that for 
$p>2$, the fusion category $\Ver_p$ has a fusion subcategory $\Ver_p^+$, such that $\Ver_p=\Ver_p^+\boxtimes sVect$,
where $sVect$ is the category of finite dimensional super vector spaces over $\mathbf{k}$.

The Drinfeld center (see e.g. \cite[7.13]{EGNO}) is a well known construction, which produces a braided tensor
category $\mathscr{Z}(\C)$ starting from an arbitrary tensor category $\C$. It is easy to see that
$\Z(\Ver_p)=\Z(\Ver_p^+)\boxtimes \Z(sVect)$, so that the computation of $\Z(\Ver_p)$ reduces to the 
computation of $\Z(\Ver_p^+)$. 

Since the fusion category $\Ver_p^+$ is symmetric, we have an injective tensor functor
$\Ver_p^+\xrightarrow{1:1} \Z(\Ver_p^+)$. In Theorem \ref{main result} we show that the Drinfeld center $\Z(\Ver_p^+)$ has the following unusual property:
Any simple object of $\Z(\Ver_p^+)$ is contained in $\Ver_p^+\subset \Z(\Ver_p^+)$. In particular,
the category $\Z(\Ver_p^+)$ has the Chevalley property; that is, the semisimple objects are closed under
the tensor product. 
Thus, one can think about the category $\Z(\Ver_p^+)$ as a ``non-semisimple fattening'' of the category
$\Ver_p^+$. In particular, the square of the braiding is unipotent in the category $\Z(\Ver_p^+)$ (see Theorem \ref{main result});
this is an unusual property for a factorizable (see \cite[8.6]{EGNO}) braided tensor category such as $\Z(\Ver_p^+)$.

The focus of this paper is to understand the ``non-semisimple'' nature of the category $\Z(\Ver_p^+)$.
Here are our main results:

In Theorem \ref{graded Cartan matrix} we compute the graded Cartan matrix of $\Z(\Ver_p^+)$, and in Theorem \ref{idealofrelations} we compute the {\bf bound quiver} of the category $\Z(\Ver_p^+)$. In Theorem \ref{cohomology ring} we compute the cohomology; that is, Ext-algebra of the unit object of the category $\Z(\Ver_p^+)$ (this computation and its generalization was independently obtained by A. Davydov and P. Etingof). It is interesting that the result does not depend on $p\ge 5$ (and the case $p=3$
is trivial since the category $\Z(\Ver_3^+)$ is equivalent to the category of finite dimensional vector spaces).
In Theorem \ref{reptype7} we prove that the category $\mathscr{Z}(\Ver_p^+)$ is wild for every $p\ge 7$.

In Corollary \ref{DB} we show that $\Z(\Ver_5^+)$ has exactly $10$ non-isomorphic indecomposable objects; we describe them in Theorem \ref{complete list} and the Green ring of $\Z(\Ver_5^+)$ in \S\ref{sec:The Green ring}, and compute the ``semisimplification'' of $\Z(\Ver_5^+)$ in Theorem \ref{semisimp}. Even though the category
$\Z(\Ver_p^+)$ is wild for $p\ge 7$, still it might be interesting to compute the semisimplifications of some pieces
of this category.

\medskip
\noindent
{\bf Acknowledgements.} 
We are very grateful to Cris Negron, Dmitry Nikshych, Julia Plavnik, and Siddharth Venkatesh
for their participation in an early stage of this project. We are also very grateful to Pavel Etingof (and his ChatGPT)
and Alexander Kleshchev for helpful discussions.

\section{Preliminaries}\label{sec:Preliminaries}

Throughout $\mathbf{k}$ is an algebraically closed field of characteristic $p>0$.

\subsection{Bound quivers associated to algebra objects}\label{sec:Quivers associated to algebra objects}
Let $\C$ be a fusion category over $\mathbf{k}$, and let $\mathcal{O}(\C)$ be a complete set of representatives of the simple objects of
$\C$. Let $\Fun(\C,\C)$
be the tensor category of $\mathbf{k}$-linear endofunctors of $\C$, with tensor
product given by composition. 

Consider the tensor functor
\begin{equation}\label{thefunF}
F:\C\rightarrow\Fun(\C,\C),
\quad
X\mapsto X\otimes-.
\end{equation}
Choose a generator
\[
P:=\bigoplus_{L\in\mathcal{O}(\C)}L,
\]
and let
\[
S:=\End(P).
\]
Since $\C$ is semisimple, $S\cong \mathbf{k}^n$, where $n:=|\mathcal{O}(\C)|$.

Recall that evaluation on $P$ induces an equivalence of tensor
categories
\begin{equation}\label{theequivF}
\Fun(\C,\C)\simeq {\rm Bimod}(S),
\end{equation}
where ${\rm Bimod}(S)$ is the category of finite dimensional
$S$-bimodules.
Since \(S\cong \mathbf{k}^n\), every $S$-bimodule $M$ decomposes as
\[
M=\bigoplus_{i,j}e_iMe_j,
\]
where \(e_1,\dots,e_n\) are the primitive idempotents of \(S\).
Thus, an \(S\)-bimodule is simply a collection of vector spaces indexed by
ordered pairs of simple objects.

It follows (see e.g. \cite[2.1]{MOV}) that associative algebra objects $A$ in ${\rm Bimod}(S)$ are therefore equivalent to finite
bound quivers (= quivers with relations); that is, the vector space
$e_iAe_j$ is the space of arrows from vertex \(j\) to vertex \(i\), and the
multiplication of \(A\) determines the relations.

\begin{proposition}
\label{prop:quiver-monad}
Let $A\in \C$ be an associative algebra.
Then the associative algebra
$F(A)\in\Fun(\C,\C)$ 
corresponds, under the equivalence (\ref{theequivF})
\[
\Fun(\C,\C)\simeq {\rm Bimod}(S),
\]
to the bound quiver describing the category ${\rm Mod}_{\C}(A)$ of left
\(A\)-modules in \(\C\).
\end{proposition}

\begin{proof}
The tensor functor $F(X)=X\otimes - $ (\ref{thefunF})
identifies the algebra object \(A\) with the monad
\[
A\otimes -:\C\rightarrow \C.
\]
By definition, the category ${\rm Mod}_{\C}(A)$ of left \(A\)-modules in \(\C\) is precisely the category of
modules over this monad. Moreover, under the equivalence (\ref{theequivF})
\[
\Fun(\C,\C)\simeq {\rm Bimod}(S),
\]
the monad $F(A)\in {\rm Bimod}(S)$ is an associative algebra, hence corresponds to a finite dimensional basic algebra, which may be presented by a bound quiver, and modules over $F(A)$ are
exactly modules over the monad $A\otimes-$, i.e. left $A$-modules in $\C$.
\end{proof}

\subsection{The fundamental group}\label{sec:The fundamental group}
Let $\mathscr{C}$ be a {\bf symmetric} fusion category over $\mathbf{k}$, and let $\mathcal{O}(\mathscr{C})$ be a complete set of representatives of the simple objects of
$\C$. Recall that in $\mathscr{C}$ we have a canonical finite group scheme $\pi(\mathscr{C})$ called the {\em fundamental group} \cite{D}. Namely, given a commutative algebra $A$ in ${\rm Ind}(\mathscr{C})$, let ${\rm Mod}(A)$ be the category of $A$-modules in ${\rm Ind}(\mathscr{C})$, and let 
$$F_A:\mathscr{C}\to {\rm Mod}(A),\quad X\mapsto A\ot X,$$
be the free module functor. Then the functor of points
$$\pi:{\rm CA}({\rm Ind}(\mathscr{C}))\to {\rm Grp},\quad A\mapsto \pi(A):={\rm Aut}^{\ot}(F_A),$$
is representable by the commutative Hopf algebra $\mathscr{O}(\pi(\mathscr{C}))$ in $\mathscr{C}$; that is,
$${\rm Aut}^{\ot}(F_A)=\Hom_{{\rm CA}({\rm Ind}(\mathscr{C}))}(\mathscr{O}(\pi(\mathscr{C})),A).$$
It is known that $\mathscr{O}(\pi(\mathscr{C}))=\bigoplus_{L_i\in \mathcal{O}(\mathscr{C})} L_i\ot L_i^*$ as an object of $\mathscr{C}$, equipped with the  product determined by the obvious morphisms 
$$L_i\otimes L_i^*\ot L_j\otimes L_j^*\to (L_i\otimes L_j)\otimes (L_i\otimes L_j)^*,\quad L_i,L_j\in \mathcal{O}(\mathscr{C}),$$
and coproduct determined by the obvious morphisms 
$$L_i\otimes L_i^*\to (L_i\otimes L_i^*)\otimes (L_i\otimes L_i^*),\quad L_i\in \mathcal{O}(\mathscr{C}).$$
In particular, we have
\begin{equation}\label{fpdimreg}
{\rm FPdim}(\mathscr{O}(\pi(\mathscr{C})))={\rm FPdim}(\mathscr{C}).
\end{equation}

Recall that for each $X\in \mathscr{C}$, the functor of points
$${\rm GL}(X):{\rm CA}({\rm Ind}(\mathscr{C}))\to {\rm Grp},\quad A\mapsto {\rm GL}(X)(A):={\rm Aut}_A(A\ot X),$$
is representable by the commutative Hopf algebra $\mathscr{O}({\rm GL}(X))$ in ${\rm Ind}(\mathscr{C})$. 
Let 
\begin{equation}\label{canac}
\epsilon:\pi\to {\rm GL}(X)
\end{equation}
be the natural transformation defined for $A\in {\rm CA}({\rm Ind}(\mathscr{C}))$ by 
$$\epsilon_A:{\rm Aut}^{\ot}(F_A)\to {\rm Aut}_A(A\ot X),\quad \theta\mapsto \theta_X.$$ 
Then $\epsilon$ is a morphism of affine algebraic group schemes in ${\rm Ind}(\mathscr{C})$, referred to as the {\em natural action of $\pi(\mathscr{C})$ on $X$}.

More generally, let $G$ be a finite group scheme in $\mathscr{C}$, equipped with a group scheme homomorphism $F:\pi(\mathscr{C})\to G$. Let 
$$\mathscr{E}:=\Rep_{\mathscr{C}}(G,F)={\rm Corep}(\O(G),F)_{\mathscr{C}}$$
be the finite symmetric tensor category of left representations of $G$ (equivalently, right $\O(G)$-corepresentations) in $\mathscr{C}$, such that the restriction to $\pi(\mathscr{C})$ via $F$ is the same as the natural action of $\pi(\mathscr{C})$. As $\O(G)$ is a commutative Hopf algebra in $\mathscr{C}$, it is known that the category ${\rm Mod}_{\mathscr{E}}(\O(G))$ of left $\O(G)$-modules in $\mathscr{E}$ forms a finite braided category, with tensor product induced by the tensor product in $\mathscr{E}$ and the Hopf structure on $\O(G)$.
By definition, an object of ${\rm Mod}_{\mathscr{E}}(\O(G))$ is a triple $(X,\rho_X,\mu_X)$, where $X$ is an object of $\mathscr{C}$, equipped with compatible right coaction $\rho_X:X\to X\ot \O(G)$ and left action $\mu_X:\O(G)\ot X\to X$ morphisms in $\mathscr{C}$.

Let $\mathscr{Z}(\mathscr{E})$ be the Drinfeld center of $\mathscr{E}$.

\begin{proposition}\label{etsn}\cite{ES,LZ}
Fix a pair $(G,F)$, and let $\mathscr{E}:=\Rep_{\mathscr{C}}(G,F)$ be as above. 
There is a natural braided equivalence 
$$E:{\rm Mod}_{\mathscr{E}}(\O(G))\simeq \mathscr{Z}(\mathscr{E}).$$
In particular, there is a natural braided equivalence
$${\rm Mod}_{\mathscr{C}}(\O(\pi(\mathscr{C})))\simeq \mathscr{Z}(\mathscr{C}).$$
\end{proposition}

\begin{proof}
Let $(X,\rho_X,\mu_X)$ be an object of ${\rm Mod}_{\mathscr{E}}(\O(G))$. For each object $(Y,\rho_Y)\in \mathscr{E}$ with right coaction morphism $\rho_Y:Y\to Y\ot \O(G)$, define the $\mathscr{E}$-morphism 
$${\rm c}_{(X,\rho_X),(Y,\rho_Y)}:(X,\rho_X)\ot (Y,\rho_Y)\to (Y,\rho_Y)\ot (X,\rho_X)$$ 
to be the composition
$$X\ot Y\xrightarrow{\id_X\ot \rho_Y} X\ot Y\ot \O(G)\xrightarrow{c_{1,23}} Y\ot \O(G)\ot X\xrightarrow{\id_Y\ot \mu_X} Y\ot X,$$
where $c_{1,23}$ is given by the symmetry of $\mathscr{C}$. Then 
$$E(X,\rho_X,\mu_X):=((X,\rho_X),{\rm c}_{(X,\rho_X),-})$$
is an object of $\mathscr{Z}(\mathscr{E})$, so we have defined a functor 
$$E:{\rm Mod}_{\mathscr{E}}(\O(G))\to \mathscr{Z}(\mathscr{E}).$$ 

In the other direction, let $((X,\rho_X),{\rm c}_{(X,\rho_X),-})$ be an object of $\mathscr{Z}(\mathscr{E})$, so that $(X,\rho_X)$ is an object of $\mathscr{E}$ and 
$$\{{\rm c}_{(X,\rho_X),(Y,\rho_Y)}:(X,\rho_X)\ot (Y,\rho_Y)\xrightarrow{\cong} (Y,\rho_Y)\ot (X,\rho_X)\mid (Y,\rho_Y)\in \mathscr{E}\}$$
is the half-braiding. Then
$$\mu_X:\O(G)\ot X\xrightarrow{{\rm c}^{-1}_{(X,\rho_X),\O(G)}}X\ot \O(G)\xrightarrow{\id_X\ot \varepsilon}X\ot \1=X$$ 
is a left action $\mathscr{E}$-morphism, so we have defined a functor 
$$\widetilde{E}:\mathscr{Z}(\mathscr{E})\to {\rm Mod}_{\mathscr{E}}(\O(G)).$$
Moreover, $E$ and $\widetilde{E}$ are braided functors, which are inverse of each other.

Finally, the last claim follows from $\mathscr{C}=\Rep_{\mathscr{C}}(\pi(\mathscr{C}),\id)$.
\end{proof}

\subsection{The Verlinde category $\Ver_p$}\label{sec:The category Verp}
Consider the symmetric fusion category $\Ver_p$ given by the semisimplification of the symmetric tensor category ${\rm Rep}(\mathbb{Z}/p\mathbb{Z})$ over $\mathbf{k}$ \cite{O}. This category has $p-1$ simples $L_1,\dots, L_{p-1}$ and $\mathfrak{sl}(2)$-like fusion rules 
\[
L_1=\1,\ L_2\ot L_i=L_{i-1}\oplus L_{i+1}.
\]
(Here $L_0=0$.)  This category splits $\Ver_p=\Ver_p^+\boxtimes sVect$, where $\Ver_p^+$ has simples $L_i$ with $i$ odd, and $sVect$ is the copy of the symmetric fusion category of super vector spaces generated by the unit and $L_{p-1}$.  For example, $\Ver_5^+$ is the Fibonacci fusion category $\Ver_5^+=\langle \1,L_3\rangle$, where the simple $L_3$ satisfies $L_3\ot L_3=\1\oplus L_3$.

\subsection{The fundamental group of $\Ver_p^+$}\label{sec:The fundamental group of Verp+}
Recall that the Lie algebra $\mathfrak{g}\mathfrak{l}(L_2)$, and its nonzero part Lie subalgebra $\mathfrak{s}\mathfrak{l}(L_2)\subset \mathfrak{g}\mathfrak{l}(L_2)$, are defined by 
$$\mathfrak{g}\mathfrak{l}(L_2)=L_2\ot L_2=\1\oplus L_3,\quad\mathfrak{s}\mathfrak{l}(L_2)=L_3,$$
and that $\mathfrak{s}\mathfrak{l}(L_2)$ is the Lie algebra of the finite connected simple group scheme $SL(L_2)$ in $\Ver_p^+$, with 
$SL(L_2)_0={\rm Spec(\mathbf{k})}$ and $\O(SL(L_2))=(U(\mathfrak{s}\mathfrak{l}(L_2)))^*\in \Ver_p^+$ \cite{V,V2}.

Consider the symmetric algebra ${\rm Sym}(L_3)=\bigoplus_{i=0}^{\infty}S^{i}(L_3)$ in $\Ver_p^+$; it is a graded commutative Hopf algebra with $S^{0}(L_3)=\1$ and $S^{1}(L_3)=L_3$. Recall \cite{EOV,V,V2} that $S^{i}(L_3)=0$ for $i>p-3$, $S^{i}(L_3)\oplus S^{i-1}(L_3)=L_{i+1}\ot L_{i+1}$ for $1\le i\le p-2$, so that we have  
\begin{equation}\label{symmalg}
{\rm Sym}(L_3)=\bigoplus_{i=0}^{p-3}S^{i}(L_3).
\end{equation} 
In particular, $S^{p-3}(L_3)=\1$ and $S^{p-4}(L_3)=L_3$.

\begin{lemma}\label{haH}
Let $H:=\mathscr{O}(\pi(\Ver_p^+))$ be the function algebra of the fundamental group of $\Ver_p^+$. The following hold:
\begin{enumerate}
\item
${\rm FPdim}(H)={\rm FPdim}(\Ver_p^+)=\frac{1}{2}\sum_{i=1}^{p-1} {\rm FPdim}(L_i)^2 = \frac{p}{4\sin^2(\pi/p)}$.
\item
$\pi(\Ver_p^+)\cong SL(L_2)$ as group schemes in $\Ver_p^+$. In particular, $\pi(\Ver_p^+)$ is connected; that is, $H$ is a local commutative Hopf algebra in $\Ver_p^+$, where locality means that the kernel $\mathfrak{m}$ of the counit morphism $\varepsilon:H\to \1$ is the unique maximal ideal in $H$.
\item
There is an isomorphism $H\cong{\rm Sym}(L_3)$ of algebras with counit. In particular, under the identification $H={\rm Sym}(L_3)$, we have that $\mathfrak{m}=\bigoplus_{i=1}^{p-3} S^{i}(L_3)$.
\end{enumerate}
\end{lemma}

\begin{proof}
(1) This follows from (\ref{fpdimreg}).

(2) This follows from \cite[Theorem 4.9]{V}.

(3) Since $\overline{H}=\mathbf{k}$ and $\mathfrak{g}:={\rm Lie}(\pi(\Ver_p^+))=\mathfrak{s}\mathfrak{l}(L_2)=L_3$, it follows from \cite[Proposition 7.20]{V2} that $H\cong{\rm Sym}(\mathfrak{g}^*)={\rm Sym}(L_3)$ as algebras with counit.
\end{proof}

\begin{example}\label{example5}
Let $p=5$, and set $X:=L_3$. We have 
$$H={\rm Sym}(X)=\1\oplus X\oplus \1.$$
The maximal ideal $\mathfrak{m}\subset H$ is the sum of $X$ and (the second copy of) $\1$, and the multiplication map
\[
m:X\ot X\to \1,
\]
which is a map in $\Ver_5^+$ is, up to a scaling, the composite of the identification $X\ot X= \1\oplus X$ with the projection $\1\oplus X\twoheadrightarrow \1$. \qed
\end{example} 

\begin{remark}\label{remarkpge7}
Let $p\ge 7$. Recall that ${\rm c}_{L_3,L_3}:\1\oplus L_3\oplus L_5\xrightarrow{\cong}\1\oplus L_3\oplus L_5$ is given by $(-1)^{3-1}{\rm id}_{\1}\oplus (-1)^{3-2}{\rm id}_{L_3}\oplus (-1)^{3-3}{\rm id}_{L_5}$. Thus, we have
$${\rm c}_{L_3,L_3}={\rm id}_{\1}\oplus (-1){\rm id}_{L_3}\oplus {\rm id}_{L_5}:\1\oplus L_3\oplus L_5\xrightarrow{\cong}\1\oplus L_3\oplus L_5
.$$
In particular, $S^{2}(L_3)=(L_3\ot L_3)/({\rm id}-{\rm c}_{L_3,L_3})=\1\oplus L_5$. \qed
\end{remark}

\begin{example}\label{example7}
Let $p=7$, and set $X:=L_3$ and $Y:=L_5$. By Remark \ref{remarkpge7}, we have
$$H={\rm Sym}(X)=\1\oplus X\oplus (\1\oplus Y)\oplus X\oplus \1,$$
the maximal ideal of $H$ is 
$$\mathfrak{m}=S^{1}(X)\oplus S^{2}(X)\oplus S^{3}(X)\oplus S^{4}(X)=X\oplus (\1\oplus Y)\oplus X\oplus \1,$$
and the multiplication map
\[
m:S^{1}(X)\ot S^{1}(X)\to S^{2}(X)
\]
is the composite of the identification $X\ot X= \1\oplus X\oplus Y$ with the projection $\1\oplus X\oplus Y\twoheadrightarrow \1\oplus Y$, and
\[
m:S^{1}(X)\ot S^{2}(X)= X\ot (\1\oplus Y)=X\oplus X\oplus Y\to S^{3}(X)=X
\]
is the projection on the first copy of $X$. \qed
\end{example}

\section{The Drinfeld center $\mathscr{Z}(\Ver_p)$}

Let $\mathscr{Z}(\Ver_p)$ be the Drinfeld center of $\Ver_p$. Clearly, $\mathscr{Z}(\Ver_2) = {\rm Vect} $, and $\mathscr{Z}(\Ver_3)=\mathscr{Z}(sVect)$ is the non-degenerate fusion category $\Rep(\mathbb{Z}/2\mathbb{Z}\times \mathbb{Z}/2\mathbb{Z})$ with braiding given by a non-degenerate form on the character group of $(\mathbb{Z}/2\mathbb{Z})^2$. 

\begin{lemma}\label{nonss}
The following hold:
\begin{enumerate}
\item
The canonical braided tensor functor
\[
\mathscr{Z}(\Ver_p^+)\boxtimes \mathscr{Z}(sVect)\to \mathscr{Z}(\Ver_p)
\]
is an equivalence of braided tensor categories.
\item
For any prime $p > 3$, $\mathscr{Z}(\Ver_p)$ and $\mathscr{Z}(\Ver_p^+)$ are not semisimple.
\end{enumerate}
\end{lemma}

\begin{proof}
(1) This follows from a dimension count.

(2) Since $\dim(L_i) = i$ for all $1\le i\le p-1$, it follows that $$\dim(\Ver_p) = \sum_{i = 1}^{p-1} i^2 = \frac{p(p-1)(2p-1)}{6}.$$
Thus, $\dim(\Ver_p)= 0 \pmod{p}$, so $\mathscr{Z}(\Ver_p)$ is not semisimple, and hence by (1), so is $\mathscr{Z}(\Ver_p^+)$.
\end{proof}

Thus, in analyzing the center $\mathscr{Z}(\Ver_p)$ for $p\ge 5$ the factor $\mathscr{Z}(sVect)$ contributes nothing interesting, and understanding $\mathscr{Z}(\Ver_p)$ is reduced to understanding the center of the positive part $\mathscr{Z}(\Ver_p^+)$.

\subsection{The category $\mathscr{Z}(\Ver_p^+)$, $p\ge 5$}\label{sec:The category Verplus} 
Let $H:={\rm Sym}(L_3)$ as in Lemma \ref{haH}. Recall that $H$ is a commutative Hopf algebra in $\Ver_p^+$, so that the category 
$${\rm Mod}(H):={\rm Mod}_{\Ver_p^+}(H)$$
of $H$-modules in $\Ver_p^+$ forms a finite braided category, with tensor product $\ot$ induced by the tensor product in $\Ver_p^+$ and the Hopf structure on $H$. Let
\begin{equation}\label{forfunf}
\widetilde{F}:{\rm Mod}(H)\twoheadrightarrow \Ver_p^+,\qquad 
F:\mathscr{Z}(\Ver_p^+)\twoheadrightarrow \Ver_p^+
\end{equation}
be the forgetful functors, and let
\begin{equation}\label{adjforfunf}
\widetilde{I}:\Ver_p^+\xrightarrow{1:1} {\rm Mod}(H),\qquad I:\Ver_p^+\xrightarrow{1:1} \mathscr{Z}(\Ver_p^+)
\end{equation}
be the adjoints of $\widetilde{F}$ and $F$, respectively. Note that $\widetilde{I}$ is the free module functor
$$\widetilde{I}:\Ver_p^+\xrightarrow{1:1} {\rm Mod}(H),\quad V\mapsto H\ot V.$$

\begin{theorem}\label{main result}
The following hold:
\begin{enumerate}
\item
There is an equivalence of braided tensor categories 
$$E:{\rm Mod}(H)\xrightarrow{\simeq}\mathscr{Z}(\Ver_p^+)$$
for which the diagram
\[
\begin{tikzcd}
  {\rm Mod}(H) \arrow{rr}{E} \arrow[swap]{dr}{\widetilde{F}} & & \mathscr{Z}(\Ver_p^+) \arrow{dl}{F} \\[10pt]
    & \Ver_p^+ 
\end{tikzcd}
\]
commutes.
\item
The $L_i$ for $i$ odd, $1\le i\le p-2$, are the simple objects of $\mathscr{Z}(\Ver_p^+)$. In particular, $\mathscr{Z}(\Ver_p^+)$ has the Chevalley property and $\Ver_p^+$ is the semisimple part of $\mathscr{Z}(\Ver_p^+)$. In other words, 
the section $\Ver_p^+\to \mathscr{Z}(\Ver_p^+)$ provided by the symmetry on $\Ver_p^+$ is an equivalence onto the fusion subcategory of semisimple objects in $\mathscr{Z}(\Ver_p^+)$.
\item
For any $V\in \Ver_p^+$, the free module $H\ot V$ is both the projective cover and injective envelope of $V$ in ${\rm Mod}(H)$. Furthermore, the $P_i:=H\ot L_i$ for $i$ odd, $1\le i\le p-2$, provide a complete listing of the indecomposable projectives in ${\rm Mod}(H)$, and we have $I(L_i)\cong E(H\ot L_i)=E(\widetilde{I}(L_i))$.
\item
The braiding ${\rm c}$ on $\mathscr{Z}(\Ver_p^+)$ is unipotent; that is, for every $V,W\in \mathscr{Z}(\Ver_p^+)$, the operator $\id-{\rm c}^2_{V,W}\in \End(V\ot W)$ is nilpotent. 
\end{enumerate}
\end{theorem}

\begin{proof}
(1) This follows from Proposition \ref{etsn} and the preceding remarks.

(2) Note that any object $V$ in $\Ver_p^+$ becomes an object in ${\rm Mod}(H)$, denoted by $T(V)$, via the trivial $H$-action,
\[
H\ot V\xrightarrow{\varepsilon\ot {\rm id}}\1\ot V\cong V,
\]
so that we have a functor
$$T:\Ver_p^+\xrightarrow{1:1} {\rm Mod}(H),\quad V\mapsto T(V).$$
Clearly, $T$ is a tensor embedding, which is a section of the forgetful functor $\widetilde{F}$.
Thus by (1), it suffices to prove that all of the objects in the image of $T$ are semisimple, and this image is precisely the (fusion) subcategory of semisimples in ${\rm Mod}(H)$. To this end, note that since $\widetilde{F}:{\rm Mod}(H)\twoheadrightarrow \Ver_p^+$ is faithful, $L_i=\widetilde{F}(T(L_i))$ is simple for any simple $L_i\in \Ver_p^+$. Moreover, by considering the $\mathfrak{m}$-adic filtration of objects in ${\rm Mod}(H)$, we see that locality of $H$ implies that for any simple $L_i\in \Ver_p^+$, there is a unique simple in ${\rm Mod}(H)$ which maps to $L_i$ under $\widetilde{F}$.

(3) It is clear that for any $V\in \Ver_p^+$, the free module $H\ot V$ is both projective and injective in ${\rm Mod}(H)$ (projectivity follows from adjunction to the exact forgetful functor,
while injectivity follows because finite tensor categories are Frobenius), and $V$ is a quotient of $H\ot V$.

Set $P_i:=P(L_i)$. Since $I(L_i)=I(\1)\ot L_i$ is projective in $\mathscr{Z}(\Ver_p^+)$, and 
$$\Hom_{\mathscr{Z}(\Ver_p^+)}(I(L_i),L_i) \cong \Hom_{\Ver_p^+}(L_i,L_i),$$
it follows that $P_i$ is a direct summand of $I(L_i)$ for each $i$. Thus, we have
\begin{equation*}
\begin{split}
& {\rm FPdim}(\Ver_p^+)^2={\rm FPdim}(\mathscr{Z}(\Ver_p^+))=\sum_i{\rm FPdim}(P_i){\rm FPdim}(L_i)\\
& \le \sum_i{\rm FPdim}(I(L_i)){\rm FPdim}(L_i)=\sum_i{\rm FPdim}(I(\mathbf{1})){\rm FPdim}(L_i)^2\\
& = {\rm FPdim}(\Ver_p^+){\rm FPdim}(\Ver_p^+)={\rm FPdim}(\Ver_p^+)^2,
\end{split}
\end{equation*}
which implies that we must have $P_i=I(L_i)$. Finally, since by Lemma \ref{haH}(1), $H$ and $I(\1)$ have the same Frobenius-Perron dimension, we have $I(\1)\cong E(H)$.

(4) Let $V_i$, $0\le i\le n$, and $W_j$, $0\le j\le m$, be the successive semisimple quotients of the Loewy filtrations of $V$ and $W$. Since by Theorem \ref{main result}, $\mathscr{Z}(\Ver_p^+)$ has the Chevalley property, $V\ot W$ has a filtration with successive semisimple quotients $\bigoplus_{i+j=k} V_i\ot W_j$, $0\le k\le n+m$. Since $\id -{\rm c}_{V_i,W_j}^2=0$ for each $i,j$, it follows from naturality of the braiding that some power of $\id-{\rm c}^2_{V,W}$ is zero.
\end{proof}

\begin{remark}
By Theorem \ref{main result}(2), $\mathscr{Z}(\Ver_p^+)$ is quasisymmetric in the sense of \cite[Definition 2.6]{EG}. \qed
\end{remark}

\subsection{The graded Cartan matrix of $\mathscr{Z}(\Ver_p^+)$}\label{sec:The Graded Cartan matrix}

Since  
\begin{equation}
L_{2a+1}\otimes L_{2c+1}
\cong
\bigoplus_{r=|a-c|}^{\min(a+c,\,p-2-a-c)}L_{2r+1},\qquad 0\leq a,b\leq \frac{p-3}{2},
\label{eq:X-fusion-rule}
\end{equation}
it follows that for $b\leq a$, $L_{2b+1}$ occurs in $L_{2a+1}\otimes L_{2c+1}$ precisely when
\begin{equation}
a-b\leq c\leq\min(a+b,\,p-2-a-b).
\label{eq:c-range}
\end{equation}

\begin{lemma}
For every $0\leq d\leq p-3$, we have
\begin{equation}
S^d(L_3)
\cong
\bigoplus_{\substack{
0\leq c\leq \min(d,p-3-d)\\
c\equiv d\; (\mathrm{mod}\,2)}}
L_{2c+1}.
\label{eq:symmetric-power-decomposition}
\end{equation}
\end{lemma}

\begin{proof}
The proof is by induction using
\begin{equation*}
S^d(L_3)\oplus S^{d-1}(L_3)
\cong L_{d+1}\otimes L_{d+1}\cong
\bigoplus_{c=0}^{\min(d,\,p-2-d)}L_{2c+1}.
\end{equation*}
The summands in
$L_{d+1}\otimes L_{d+1}$ are indexed by all integers 
$0\leq c\leq \min(d,p-2-d)$, and 
the proposed decompositions of $S^d(L_3)$ and $S^{d-1}(L_3)$ partition this
interval into its two parity classes, which proves the induction step.
\end{proof}

Recall the quantum integers
\[
[m]_q=\frac{q^m-q^{-m}}{q-q^{-1}}.
\]
For every $r=2b+1$ and $s=2m$, we have 
\begin{equation*}
[2b+1]_q[2m]_q
=
\sum_{\ell=0}^{\min(2b+1,2m)-1}
[2(m+b-\ell)]_q,
\end{equation*}
and since
$[2m]_q=[2]_q[m]_{q^2}$,
dividing by $[2]_q$ gives
\begin{equation}
[2b+1]_q[m]_{q^2}
=
\sum_{\ell=0}^{\min(2b+1,2m)-1}
[m+b-\ell]_{q^2}.
\label{eq:key-q-identity}
\end{equation}

Let $R:={\rm Gr}(\Ver_p^+)$. Let 
\[
h(t):=
\sum_{d=0}^{p-3}[S^d(L_3)]t^d
\in R[t]
\]
be the graded Grothendieck class of $H={\rm Sym}(L_3)=\bigoplus_{d=0}^{p-3}S^d(L_3)$. 
Since for each \(0\le c\le (p-3)/2\), \eqref{eq:symmetric-power-decomposition} shows that \(L_{2c+1}\) occurs in the symmetric powers of \(L_3\) precisely in degrees
\[
c,c+2,c+4,\dots,p-3-c,
\]
it follows that the graded multiplicity of $L_{2c+1}$ in \(H\) is
\begin{align}
\sum_{d=0}^{p-3}[S^d(L_3):L_{2c+1}]t^d
&=t^c+t^{c+2}+\cdots+t^{p-3-c} \notag=t^{p-3}[p-2-c]_t.
\label{eq:graded-multiplicity-Xc}
\end{align}
Therefore, setting $p=2n+1$, and
\[
a_c:=t^{n-1}[n-c]_t 
\quad
(0\leq c\leq n-1),\quad a_n:=0,
\]
we obtain that  
\begin{equation}\label{5}
h(t)=t^{n-1}\sum_{c=0}^{n-1}a_cL_{2c+1}.
\end{equation}

Now the canonical grading $H={\rm Sym}(L_3)=\bigoplus_{d=0}^{p-3}S^d(L_3)$ determines a graded lift
of every $P_i=H\otimes L_i$. Namely, $P_i$ has an $H$-module filtration by the subobjects
\[
P_i^{\geq d}:=\bigoplus_{r\geq d}S^r(L_3)\otimes L_i,
\]
with successive quotients 
\[
P_i^{\geq d}/P_i^{\geq d+1}
\cong S^d(L_3)\otimes L_i,
\]
and with the ideal $H_{>0}$ acting trivially.

Next let ${\rm C}=({\rm C}_{ij})$ be the Cartan matrix of $\Z(\Ver_p^+)$, where
$${\rm C}_{ij}=[P_i:L_j]=\dim\left(\Hom_{\Z(\Ver_p^+)}(P_j,P_i)\right)$$
for every odd 
integers $i,j\in\{1,3,\dots,p-2\}$. Let ${\rm C}(t):=({\rm C}_{ij}(t))$ be the {\bf graded} Cartan matrix, where 
\begin{equation}
{\rm C}_{ij}(t)
:=
\sum_{d=0}^{p-3}
\dim\left(\Hom
\bigl(L_j,S^d(L_3)\otimes L_i\bigr)\right)t^d.
\label{eq:graded-Cartan-definition}
\end{equation}
Then ${\rm C}_{ij}={\rm C}_{ij}(1)$ and ${\rm C}={\rm C}(1)$.

Also let \({\rm A}=({\rm A}_{ij})\) be the matrix of multiplication by \(L_3\) in \(R={\rm Gr}(\Ver_p^+)\), with respect to the ordered basis
$L_1=\1,L_3,\ldots,L_{p-2}$, so that
\[
{\rm A}_{ij}
=
\dim\operatorname{Hom}
\bigl(L_j,L_3\otimes L_i\bigr).
\]
Equivalently, \({\rm A}\) is the adjacency matrix of the quiver of $\Z(\Ver_p^+)$ (see \S\ref{sec:The quiver} below).
Then using $L_3\otimes \1=L_3$, $L_3\otimes L_{p-3}=L_{p-5}\oplus L_{p-3}$, and 
\[
L_3\otimes L_{i}
=
L_{i-2}\oplus L_{i+2}\oplus L_{i+4},
\qquad
3\le i\le (p-5)/2\text{ is odd},
\]
we have 
\begin{equation}\label{matrixA}
{\rm A}_{ij}
=
\begin{cases}
1,& |i-j|=2,\\
1,& i=j\ge 3,\\
0,&\text{otherwise}.
\end{cases}
\end{equation}

Let \({\rm M}(y)\) be the multiplication matrix of
\(y\in R\), so that
\begin{equation}\label{M}
L_iy=\sum_j {\rm M}(y)_{ij}L_j.
\end{equation}
Then ${\rm M}(yz)={\rm M}(y){\rm M}(z)$, \(y,z\in R\),
and by definition, we have
\begin{equation}\label{multiplication matrix}
{\rm A}={\rm M}(L_3),\qquad {\rm C}(t)={\rm M}(h(t)).
\end{equation}

\begin{theorem}\label{graded Cartan matrix}
Let $p\ge 3$ be prime.  The following hold:
\begin{enumerate}
\item
For every odd integers
$1\leq j\leq i\leq p-2$, we have
\begin{equation*}
{\rm C}_{ij}(t)={\rm C}_{ji}(t)
=t^{(p-3)/2}[j]_{\sqrt t}
\left[\frac{p-i}{2}\right]_t,\text{ so } {\rm C}_{ij}=
\frac{j(p-i)}{2}.
\end{equation*}
\item
$\det ({\rm C}(t))=
\left(1+t+t^2+\cdots+t^{p-1}\right)^{(p-3)/2}$, so $
\det ({\rm C})=p^{(p-3)/2}$.
\item
For every odd integers
$1\leq j\leq i\leq p-2$, we have ${\rm C}_{ij}(t)=t^{p-3}{\rm C}_{ij}(t^{-1})$.
\item
${\rm C}(t)\Bigl((t+1+t^{-1})I-{\rm A}\Bigr)=
\frac{1+t+\cdots+t^{p-1}}{t}\,I$, so ${\rm C}(3I-{\rm A})=pI$.
\item 
${\rm C}(t)=t^{(p-3)/2}\operatorname{adj}
\Bigl((t+1+t^{-1})I-{\rm A}\Bigr)$, so 
${\rm C}=\operatorname{adj}(3I-{\rm A})$.
\end{enumerate}
\end{theorem}

\begin{proof}
(1) Combining
\eqref{eq:c-range} and
\eqref{eq:graded-Cartan-definition}, 
we obtain
\begin{equation*}
{\rm C}_{ij}(t)
=t^{(p-3)/2}\sum_{i-j}^{\min(j,p-i)+i-j-1}\left[\frac{p-3}{2}-i+j+1\right]_t.
\end{equation*}
Writing $c=i-j+\ell$, we have
\[
\frac{p-3}{2}-c+1=\frac{p-i}{2}+j-\ell,
\]
hence
\begin{equation}
{\rm C}_{ij}(t)
=t^{(p-3)/2}
\sum_{\ell=0}^{\min(j,p-i)-1}
\left[\frac{p-i}{2}+b-\ell\right]_t.
\label{eq:Cartan-sum-ell}
\end{equation}
For $q=\sqrt t$, \eqref{eq:key-q-identity} and 
\eqref{eq:Cartan-sum-ell} yield
\[
{\rm C}_{ij}(t)
=t^{(p-3)/2}[j]_{\sqrt t}\left[\frac{p-i}{2}\right]_t,
\]
as desired.  

(2) Index the rows and columns by
$a,b\in\{0,1,\dots,(p-3)/2\}$, corresponding to the simple objects
$L_{2a+1}$.  Write
\[
{\rm C}(t)=t^{(p-3)/2} {\rm B}(t),
\]
where ${\rm B}=({\rm B}_{ab})$ with 
\begin{equation*}
{\rm B}_{ab}
:=[2\min(a,b)+1]_{\sqrt t}
\left[(p-3)/2-\max(a,b)+1\right]_{t}.
\end{equation*}
Then setting $x_a:=[2a+1]_{\sqrt t}$ and $y_a:=[(p-3)/2-a+1]_{t}$, we have
\[
{\rm B}_{ab}=x_{\min(a,b)}y_{\max(a,b)}.
\]
Now successive row or column elimination gives 
\begin{equation*}
\det ({\rm B})
=x_0y_{(p-3)/2}
\prod_{a=1}^{(p-3)/2}
\bigl(x_ay_{a-1}-x_{a-1}y_a\bigr).
\end{equation*}
Here, $x_0=y_{(p-3)/2}=1$, and a direct calculation with quantum integers gives
\begin{align*}
x_ay_{a-1}-x_{a-1}y_a
&=[2a+1]_{\sqrt t}[(p-3)/2-a+2]_{t}
  -[2a-1]_{\sqrt t}[(p-3)/2-a+1]_{t}\\
  &=[p]_{\sqrt t},
\end{align*}
independently of $a$.  Therefore, we have
$\det ({\rm B})=[p]_{\sqrt t}^{(p-3)/2}$.

Since ${\rm C}(t)=t^{(p-3)/2}{\rm B}(t)$ is a $((p-3)/2+1)\times((p-3)/2+1)$ matrix,
\[
\det ({\rm C}(t))
=t^{(p-3)/2((p-3)/2+1)}[p]_{\sqrt t}^{(p-3)/2}
=\left(t^{(p-3)/2+1}[p]_{\sqrt t}\right)^{(p-3)/2}.
\]
Since 
$t^{(p-3)/2+1}[p]_{\sqrt t}
=1+t+t^2+\cdots+t^{p-1}$, the claim follows.

(3) This is straightforward.

(4) Consider the element $h(t)(t+t^{-1}+1-L_3)\in R[t,t^{-1}]$, and write 
\begin{equation}\label{9}
h(t)(t+t^{-1}+1-L_3)
=:
\sum_{c=0}^{n-1}b_cL_{2c+1}.
\end{equation}
By (\ref{5}) and (\ref{9}), the coefficient of \(L_0=\1\) is
\begin{equation}\label{10}
b_0
=
(t+t^{-1}+1)a_0-a_1,
\end{equation}
while, for \(1\leq c\leq n-1\),
\begin{equation*}
b_c=(t+t^{-1})a_c-a_{c-1}-a_{c+1}.
\end{equation*}
Since the quantum integers satisfy the recurrence
\begin{equation}\label{12}
(t+t^{-1})[m]_t
=
[m+1]_t+[m-1]_t,
\end{equation}
it follows for \(m:=n-c\) that
\[
(t+t^{-1})a_c
=
a_{c-1}+a_{c+1},
\qquad
1\leq c\leq n-1.
\]
Hence,
\begin{equation}\label{13}
b_c=0,
\qquad
1\leq c\leq n-1.
\end{equation}

It remains to determine \(b_0\). By (\ref{10}) and (\ref{12}),
\begin{align*}
b_0
&=
t^{n-1}
\left(
(t+t^{-1}+1)[n]_t-[n-1]_t
\right)\\
&=
t^{n-1}
\left(
[n+1]_t+[n-1]_t+[n]_t-[n-1]_t
\right)=
t^{n-1}\bigl([n+1]_t+[n]_t\bigr).
\end{align*}
Since
\begin{align*}
[n+1]_t+[n]_t
&=
\left(t^n+t^{n-2}+\cdots+t^{-n}\right)
+
\left(t^{n-1}+t^{n-3}+\cdots+t^{-(n-1)}\right)\\
&=
t^n+t^{n-1}+\cdots+t^{-n}=
[2n+1]_{\sqrt t}=
[p]_{\sqrt t},
\end{align*}
it follows that
\begin{equation}\label{14}
b_0=t^{n-1}[p]_{\sqrt t}.
\end{equation}
Thus, since \(p=2n+1\), we have
\begin{align*}
b_0=t^{n-1}[p]_{\sqrt t}
&=
t^{n-1}
\left(t^n+t^{n-1}+\cdots+t^{-n}\right)=
t^{-1}+1+t+\cdots+t^{2n-1}\\
&=
\frac{1+t+\cdots+t^{2n}}{t}=
\frac{1+t+\cdots+t^{p-1}}{t}.
\end{align*}
Combining this with (\ref{9}), (\ref{13}) and (\ref{14}), we obtain the following
identity in the ring \(R[t,t^{-1}]\):
\begin{equation}\label{15}
h(t)(t+t^{-1}+1-L_3)
=
\frac{1+t+\cdots+t^{p-1}}{t}\,\1.
\end{equation}
Finally applying \({\rm M}\) from (\ref{M}) to (\ref{15}), and using
(\ref{multiplication matrix}), yield
\[
{\rm C}(t)\Bigl((t^{-1}+1+t)I-A\Bigr)
=
\frac{1+t+\cdots+t^{p-1}}{t}\,I,
\] 
as claimed.

(5) This follows from (4) and  
$\det\Bigl((t+1+t^{-1})I-A\Bigr)
=[p]_{\sqrt t}$.
\end{proof}

\begin{example}($p=7$)
The graded Cartan matrix is
\begin{equation*}
{\rm C}(t)=
\begin{pmatrix}
1+t^2+t^4 & t+t^3 & t^2\\[2mm]
t+t^3 & 1+t+2t^2+t^3+t^4 & t+t^2+t^3\\[2mm]
t^2 & t+t^2+t^3 & 1+t+t^2+t^3+t^4
\end{pmatrix},
\end{equation*}
the Cartan matrix is
\[
{\rm C}=
\begin{pmatrix}
3&2&1\\
2&6&3\\
1&3&5
\end{pmatrix},
\]
$\det ({\rm C}(t))=\left(1+t+t^2+\cdots+t^{6}\right)^2$ and $\det ({\rm C})=7^2$. \qed
\end{example}

\subsection{Minimal extensions of $\Ver_p^+$}\label{sec:Minimal extensions}
Let $\mathscr{D}$ be a finite {non-degenerate} braided tensor category over $\mathbf{k}$ that contains $\Ver_p^+$ as a Lagrangian subcategory, so that by \cite{S}, we have ${\rm FPdim}(\mathscr{D})={\rm FPdim}(\Ver_p^+)^2$.

\begin{question}\label{minext}
Is it true that $\mathscr{D}\simeq \mathscr{Z}(\Ver_p^+)$ as braided categories? In particular, 
would this follow if we assume that $\Ver_p^+$ is the semisimple part of $\mathscr{D}$? \qed
\end{question}

\subsection{The cohomology ring of $\mathscr{Z}(\Ver_p^+)$}\label{sec:The cohomology ring}
Set $\mathscr{Z}:=\mathscr{Z}(\Ver_p^+)$. Let
\begin{equation}\label{minprojres}
(P^{\bullet})\qquad\cdots \to P^3\xrightarrow{\partial^3} P^2\xrightarrow{\partial^2}P^1\xrightarrow{\partial^1}P^0\xrightarrow{\varepsilon}\1\to 0
\end{equation}
be a minimal projective resolution of $\1$ in $\mathscr{Z}$; that is, $P^0=P(\1)$, and for every $n\ge 1$, $P^n=P(\ker (\partial^{n-1}))$. 
Then for each simple $L_i\in \mathscr{Z}$, 
\begin{equation}\label{minprojress}
(P^{\bullet}\ot L_i)\qquad\cdots \to P^2\ot L_i\xrightarrow{\partial^2\ot \id}P^1\ot L_i\xrightarrow{\partial^1\ot \id}P^0\ot L_i\xrightarrow{\varepsilon\ot \id}L_i\to 0
\end{equation}
is a minimal projective resolution of $L_i$ in $\mathscr{Z}$. Recall that for each simple $L_j\in \mathscr{Z}$, we have
\begin{equation}\label{compextn}
{\rm Ext}^{n}(L_i,L_j)=\Hom(P^n\ot L_i,L_j).
\end{equation}
In particular, the cohomology ring of $\mathscr{Z}$ 
$$H^{\bullet}(\mathscr{Z}):=\bigoplus_{n\ge 0}{\rm Ext}^n(\1,\1)=\bigoplus_{n\ge 0}\Hom(P^n,\1)$$
is a graded commutative algebra, and ${\rm Ext}^{\bullet}(L_i,L_j):=\bigoplus_{n\ge 0}{\rm Ext}^n(L_i,L_j)$ is a graded $H^{\bullet}(\mathscr{Z})$-module for each $i,j$.

Recall that for every odd $i$, $1\le i\le p-2$, $P_i:=P(L_i)$ denotes the projective cover of $L_i$ in $\mathscr{Z}$. 

\begin{theorem}\label{cohomology ring}
The following hold:
\begin{enumerate}
\item
The unit object $\1$ (hence, every object) of $\mathscr{Z}$ admits a $4$-periodic minimal projective resolution $P^{\bullet}$, with 
$P^0=P^3=P_1$ and $P^1=P^2=P_3$.
\item
There is an isomorphism of graded commutative algebras 
$$H^{\bullet}(\mathscr{Z})\cong \mathbf{k}[x,y]/\langle y^2\rangle,$$
where $x$ sits in degree $4$ and $y$ in degree $3$.
\item
For every odd $i,j$, $1\le i,j\le p-2$, we have 
\[
{\rm Ext}^{n}(L_i,L_j)=
\left\{\begin{array}{ll}
\Hom(L_i,L_j)=\delta_{i,j}\mathbf{k}, &  n=0,3\, (\text{mod }4)\\
d_{ij}\mathbf{k}, & n=1,2\, (\text{mod }4)
\end{array}\right.,
\]
\end{enumerate}
where $d_{ij}:=\dim\Hom(L_3\ot L_i,L_j)\in\{0,1\}$.
\end{theorem}

\begin{proof}
(1) By Lemma \ref{haH}(2), $P_1=H={\rm Sym}(L_3)$, so we can consider the (graded) Koszul complex $K^{\bullet}(L_3)={\rm Sym}(L_3)\ot \Lambda^{\bullet}(L_3)=H\ot \Lambda^{\bullet}(L_3)$  
$$0\to H\xrightarrow{\partial^3}H\ot L_3\xrightarrow{\partial^2}H\ot L_3\xrightarrow{\partial^1}H\to 0$$
introduced in \cite[Definition 2.1]{E}. 
Then by \cite[Proposition 5.1]{E}, we have that $H^0(K^{\bullet}(L_3))=H^3(K^{\bullet}(L_3))=\1$ and $H^1(K^{\bullet}(L_3))=H^2(K^{\bullet}(L_3))=0$. This implies that we have an exact sequence
$$\1\to H\xrightarrow{\partial^3}H\ot L_3\xrightarrow{\partial^2}H\ot L_3\xrightarrow{\partial^1}H\xrightarrow{\varepsilon}\1,$$
so that $H$ is a $(p-3,3)$-Koszul algebra in the sense of \cite[Definition 5.3]{E}. 
Now this exact sequence can be extended periodically to give a $4$-periodic minimal projective resolution of $\1$ given by
\begin{equation}\label{perminres}
\cdots \xrightarrow{\partial^2}P_3\xrightarrow{\partial^1}P_{1}\xrightarrow{\varepsilon} P_1\xrightarrow{\partial^3}P_3\xrightarrow{\partial^2}P_3\xrightarrow{\partial^1}P_1\xrightarrow{\varepsilon}\1
\end{equation}
(as $H\ot L_3=P_1\ot L_3=P_3$), as claimed.
 
(2) Since $\Hom(P_1,\1)=\mathbf{k}$ and $\Hom(P_3,\1)=0$, the additive structure and grading follow from (1), so in particular, $y^2=0$ and $xy=yx$. The fact that $x$ has infinite order is obtained analogously to \cite[Corollary 5.6]{E}. 

(3) For every odd $i$, $1\le i\le p-2$, tensoring the $4$-periodic minimal projective resolution of $\1$ (\ref{perminres}) with $L_i$ yields a $4$-periodic minimal projective resolution of $L_i$ given by
\begin{equation}\label{perminresi}
\cdots \xrightarrow{} P_i\xrightarrow{}P_3\ot L_i\xrightarrow{}P_3\ot L_i\xrightarrow{}P_i\xrightarrow{\varepsilon_i}L_i.
\end{equation} 
Thus, the claim follows from (\ref{compextn}).
\end{proof}

\begin{remark} 
(1) It follows from Theorem \ref{cohomology ring}(3) that $\dim {\rm Ext}^{1}(L_i,L_j)={\rm A}_{ij}$ where
${\rm A}$ is the matrix defined in \eqref{matrixA}.

(2) Let us define the power series ${\rm E}_{ij}(t):=\sum_{n=0}^\infty \dim {\rm Ext}^{n}(L_i,L_j)t^n$, and matrix
${\rm E}(t):=({\rm E}_{ij}(t))$. Theorem \ref{cohomology ring}(3) gives
\[
{\rm {\rm E}_{ij}(t)}
=
\begin{cases}
\frac{1+t^3}{1-t^4}, & i=j=1,\\
\frac{1+t+t^2+t^3}{1-t^4}, & i=j>1,\\
\frac{t+t^2}{1-t^4}, & |i-j|=2,\\
0,&|i-j|>2,
\end{cases}
\,=\,\frac{t+t^2}{1-t^4}
\begin{cases}
t+t^{-1}-1, & i=j=1,\\
t+t^{-1}, & i=j>1,\\
1, & |i-j|=2,\\
0,&|i-j|>2.
\end{cases}
\]
It follows that ${\rm E}(-t)=\frac{t-t^2}{1-t^4}((t+1+t^{-1})I-{\rm A})$. By Theorem \ref{graded Cartan matrix}(4), we get the matrix identity
\begin{equation}\label{Froberg}
{\rm C}(t){\rm E}(-t)=\frac{1-t^p}{1-t^4}I.
\end{equation}
\end{remark}

\subsection{The bound quiver of $\mathscr{Z}(\Ver_p^+)$}\label{sec:The quiver} 
Specialize \S\ref{sec:Quivers associated to algebra objects} to the case $\C:=\Ver_p^+$ and $A=H={\rm Sym}(L_3)$.

\begin{corollary}\label{cor:symmetric}
Let
$\C=\Ver_p^+$ and $H=\Sym(L_3)$.
Then the bound quiver of the category
\[
\mathscr{Z}(\Ver_p^+)={\rm Mod}_{\Ver_p^+}(H)
\]
is determined by the algebra 
$F(H)$. Moreover, since
\[
\Sym(L_3)
=
T^\bullet(L_3)
/\langle\Lambda^2(L_3)\rangle,
\]
the quiver is generated by $F(L_3)$, 
all defining relations are homogeneous quadratic relations coming from the
copy of $
\Lambda^2(L_3)\subset L_3\otimes L_3$, and there is exactly one relation for each arrow in the quiver.
\end{corollary}

\begin{proof}
The algebra $H={\rm Sym}(L_3)$ is presented as a quotient of the tensor algebra $T^\bullet (L_3)$
by the ideal generated by 
$$L_3=\wedge^2(L_3)\subset T^2(L_3)=L_1\oplus L_3,\text{  for $p=5$},$$
and 
$$L_3=\wedge^2(L_3)\subset T^2(L_3)=L_1\oplus L_3\oplus L_5,\text{  for $p\ge 7$}.$$
Moreover the relations are precisely the intersection of the kernel of the canonical map
$T^\bullet (L_3)\to H$ with $T^2(L_3)$.
Thus, the algebra $F(H)$ is generated by $F(L_3)$ (in degree $1$) with relations which are also $F(L_3)$ but now 
in degree $2$; moreover we have exactly one relation for each arrow in the quiver $F(L_3)$.

Since the functor $F$ is exact (as $\C$
is semisimple) the result follows.
\end{proof}

\begin{remark} 
Corollary \ref{cor:symmetric} implies that the algebra $F(H)$ is quadratic. Note that the Hilbert series of this algebra is precisely the graded Cartan matrix from Theorem \ref{graded Cartan matrix}. Thus 
identity \eqref{Froberg} shows that the algebra $F(H)$ is not Koszul since the Fr\"oberg condition (see \cite[Theorem 2.11.1]{BGS}) fails. \qed
\end{remark}

By Corollary \ref{cor:symmetric}, the quiver $Q$ of $\mathscr{Z}(\Ver_p^+)$ is given by
\begin{equation}\label{quiverp}
\xymatrix{
\bullet_1\ar@/^.5pc/[r]^{b_{1}} & \bullet_3\ar@/^.5pc/[l]^{a_{1}}\ar@(ul,ur)^{d_{1}}\ar@/^.5pc/[r]^{b_{2}}&\bullet_5\ar@/^.5pc/[r]^{b_{3}} \ar@/^.5pc/[l]^{a_{2}}\ar@(ul,ur)^{d_{2}}
& \bullet_7 \ar@/^.5pc/[l]^{a_{3}}\ar@(ul,ur)^{d_{3}} & \cdots & \bullet_{p-4}\ar@/^.5pc/[r]^{b_{\frac{p-3}{2}}}\ar@(ul,ur)^{d_{\frac{p-5}{2}}} & \bullet_{p-2}\ar@/^.5pc/[l]^{a_{\frac{p-3}{2}}}
\ar@(ul,ur)^{d_{\frac{p-3}{2}}}}.
\end{equation}

Corollary \ref{cor:symmetric} implies the following result.

\begin{theorem}\label{reptype7}
For every prime $p\ge 7$, the category $\mathscr{Z}(\Ver_p^+)$ is wild.
\end{theorem}

\begin{proof}
Let $(Q,I)$ be the bound quiver of $\mathscr{Z}(\Ver_p^+)$. Let $J:={\rm Rad}(\mathbf{k}Q)$. Since there is an obvious surjective algebra map $\mathbf{k}Q/I\twoheadrightarrow \mathbf{k}Q/J^2$, it suffices to prove that the algebra $\mathbf{k}Q/J^2$ is wild.

Since $\mathbf{k}Q/J^2$ is a radical-square-zero basic algebra, it is sufficient to prove that the hereditary algebra $\mathbf{k}Q_s$ is wild, where $Q_s$ is the {\em separated} quiver of $Q$ (see \cite{DF,DR}). 
Since $Q_s$ is neither Dynkin nor Euclidean, it follows that $Q_s$ is wild, as desired.
\end{proof}

We now turn to describe the algebra $F(H)$, namely the admissible ideal of generators $I\subset J^2\subset \mathbf{k}Q$.
To do so, we need to

1) Choose a basis in the vector space $F(L_3)$;

2) Find one quadratic relation for each arrow in $F(H)$.

It is convenient to perform the computation in the category of tilting modules over $SL(2)$
(since there is a canonical full tensor functor from this category to $\Ver_p$, see  e.g. \cite{V}).
Thus, let $T(\lambda)$ be the indecomposable tilting $SL(2)$-module with highest weight 
$\lambda \in \mathbb{Z}_{\ge 0}$. We will need only $0\le \lambda \le p-1$ (since the image 
of such $T(\lambda)$ in $\Ver_p$ is precisely $L_{\lambda +1}$ or $0$ if $\lambda =p-1$);
moreover we can restrict ourselves to even $\lambda$ since we are interested in $\Ver_p^+$.

For $\lambda =2$, the module
$T(2)$ (corresponding to $L_3\in \Ver_p^+$) is the adjoint representation of $SL(2)$; we choose 
a standard $\mathfrak{sl}_2$-triple basis $\{ e,h,f\}$ of $T(2)$. 
For $0\le \lambda \le p-1$, the module $T(\lambda)$ is a simple highest weight module, and 
we choose a highest weight vector $v_\lambda \in T(\lambda)$, so that $ev_\lambda =0$ and
$hv_\lambda =\lambda v_\lambda$.

Any map of $SL(2)$-modules $T(\lambda)\to ?$ is completely determined by the image of the vector
$v_\lambda$ (this image must be a vector of weight $\lambda$ annihilated by $e$). We now choose the unique (up to scalar) highest weight $SL(2)$-module maps corresponding to the three summands of $T(2)\ot T(\lambda)$; that is, 
we choose the maps 
$$a=a_{\lambda -1}: T(\lambda)\to T(2)\ot T(\lambda -2),\quad v_\lambda\mapsto e\ot v_{\lambda -2}\quad (2\le \lambda \le p-3),$$
$$d=d_{\lambda}: T(\lambda)\to T(2)\ot T(\lambda),\quad v_\lambda\mapsto \lambda h\ot v_\lambda +2e\ot fv_\lambda\quad (2\le \lambda \le p-3),$$
and 
$$b=b_{\lambda +1}: T(\lambda)\to T(2)\ot T(\lambda +2),$$
$$v_\lambda\mapsto (\lambda+2)(\lambda+1)f\ot v_{\lambda +2}-(\lambda+1)h\ot fv_{\lambda +2}-e\ot f^2v_{\lambda +2}\quad (0\le \lambda \le p-5).$$

\begin{proposition} 
After choosing a nonzero basis vector in each $1$-dimensional arrow space, the bimodule $F(L_3)$ identifies with the arrow bimodule of the bound quiver attached to $\mathscr{Z}(\Ver_p^+)$. Equivalently, the degree-one generator of the algebra $F(H)$ is exactly the bimodule of arrows, and the quadratic relations come from the copy of $\Lambda^2(L_3)\subset L_3\otimes L_3$. 
\end{proposition} 

\begin{proof} 
By Corollary \ref{cor:symmetric}, the bound quiver of $\mathscr{Z}(\Ver_p^+)\cong {\rm Mod}_{\Ver_p^+}(H)$ is determined by the algebra $F(H)$, and since $H=\Sym(L_3)=T^\bullet(L_3)/\langle \Lambda^2(L_3)\rangle$, the quiver is generated in degree $1$ by $F(L_3)$, with homogeneous quadratic relations coming from the copy of $\Lambda^2(L_3)\subset L_3\otimes L_3$. In particular, the arrow bimodule is the degree-one piece of $F(H)$, namely $F(L_3)$. Now working with the adjoint tilting module $T(2)$ corresponding to $L_3$, and for each even highest weight $\lambda$ choosing a highest weight vector $v_\lambda\in T(\lambda)$, makes it explicit. The three direct summands in 
$$T(2)\otimes T(\lambda)\cong T(\lambda-2)\oplus T(\lambda)\oplus T(\lambda+2)$$
are then represented by the maps 
$a,b,d$ chosen above. These are exactly the chosen basis vectors in the three $1$-dimensional Hom-spaces corresponding to the three summands. Thus, after fixing the basis in each direct summand, the endofunctor $F(L_3)$ decomposes into the arrow spaces of the quiver, with the three components giving the arrows $a,b,d$ (and their translates at the other vertices). This is precisely the statement that $F(L_3)$ identifies with the arrow bimodule. 
\end{proof}

\begin{theorem}\label{idealofrelations}
The admissible ideal of relations $I\subset J^2$ for $\mathscr{Z}(\Ver_p^+)$ is generated by the following relations:
\begin{equation} \label{qrel1}
d_{\lambda-2} a_{\lambda -1}-a_{\lambda -1} d_\lambda\quad (\lambda =2,4,\dots ,p-3),
\end{equation}
\quad 
so $a_1d_2\in I$, 
\begin{equation} \label{qrel2}
d_{\lambda+2} b_{\lambda +1}-b_{\lambda +1} d_\lambda\quad (\lambda =0,2,\dots ,p-5),
\end{equation}
so $d_2b_1\in I$,
and 
\begin{equation} \label{qrel3}
d_{\lambda}^2-\frac{\lambda^2}{\lambda+1}a_{\lambda +1}b_{\lambda +1}+\frac{(\lambda+2)^2}{\lambda+1}b_{\lambda -1} a_{\lambda -1}\quad (\lambda =2,4,\dots ,p-3),
\end{equation}
so $d_{p-3}^2-\frac12b_{p-4}a_{p-4}\in I$.  
Note that \eqref{qrel1}-\eqref{qrel2} can be restated as follows: the element $d:=d_2+d_4+\ldots +d_{p-3}$ is central.

Thus, there is an equivalence $\mathscr{Z}(\Ver_p^+)\simeq {\rm Mod}(\mathbf{k}Q/I)$ of Abelian categories, where $Q$ is the quiver (\ref{quiverp}).
\end{theorem}

\begin{proof}
We compute some quadratic relations between the maps $a, b, d$ extended
to maps $T(\lambda)\to S^2(T(2))\ot T(\lambda)$; the images of these relations in ${\rm Fun}(\Ver_p^+,\Ver_p^+)$
will give us the relations in the bound quiver.

First we compute
$$da(v_\lambda)=d(e\ot v_{\lambda-2})=e\ot (\lambda -2)h\ot v_{\lambda -2}+2e\ot e\ot fv_{\lambda -2},$$
and
\begin{eqnarray*}
\lefteqn{
ad(v_\lambda)=a(\lambda h\ot v_\lambda +2e\ot fv_\lambda)}\\
& = &\lambda h\ot e\ot v_{\lambda -2}+2e\ot f(e\ot v_{\lambda -2})\\
& = &
\lambda h\ot e\ot v_{\lambda -2}+2e\ot (-h)\ot v_{\lambda -2}+2e\ot e \ot fv_{\lambda -2}.
\end{eqnarray*}
Thus, 
$$d_{\lambda-2} a_{\lambda -1}(v_\lambda)=a_{\lambda -1} d_\lambda(v_\lambda)\quad (2\le \lambda \le p-3).$$

Second we compute 
\begin{eqnarray*}
\lefteqn{
db(v_\lambda)=d((\lambda+2)(\lambda+1)f\ot v_{\lambda +2}-(\lambda+1)h\ot fv_{\lambda +2}-e\ot f^2v_{\lambda +2})}\\
& = & (\lambda+2)(\lambda+1)f\ot ((\lambda+2)h\ot v_{\lambda +2}+2e\ot fv_{\lambda+2})\\
& - & (\lambda+1)h\ot f((\lambda+2)h\ot v_{\lambda +2}+2e\ot fv_{\lambda+2})\\
& - & e\ot f^2((\lambda+2)h\ot v_{\lambda +2}+2e\ot fv_{\lambda+2})
\end{eqnarray*}
and
\begin{eqnarray*}
\lefteqn{
bd(v_\lambda)=b(\lambda h\ot v_\lambda +2e\ot fv_\lambda)}\\
& = & \lambda h\ot ((\lambda+2)(\lambda+1)f\ot v_{\lambda +2}-(\lambda+1)h\ot fv_{\lambda +2}-e\ot f^2v_{\lambda +2})\\
& + & 2e\ot f((\lambda+2)(\lambda+1)f\ot v_{\lambda +2}-(\lambda+1)h\ot fv_{\lambda +2}-e\ot f^2v_{\lambda +2}).
\end{eqnarray*}
Thus, 
$$d_{\lambda+2} b_{\lambda +1}(v_\lambda)=b_{\lambda +1} d_\lambda(v_\lambda)\quad (0\le \lambda \le p-5).$$

Finally we compute
\begin{eqnarray*}
\lefteqn{ab(v_\lambda)=a((\lambda+2)(\lambda+1)f\ot v_{\lambda +2}-(\lambda+1)h\ot fv_{\lambda +2}-e\ot f^2v_{\lambda +2})}\\
& = &(\lambda+2)(\lambda+1)f\ot e\ot v_{\lambda}-(\lambda+1)h\ot f(e\ot v_{\lambda})-e\ot f^2(e\ot v_{\lambda}),
\end{eqnarray*}
$$ba(v_\lambda)=b(e\ot v_{\lambda -2})=e\ot (\lambda (\lambda-1)f\ot v_{\lambda}-(\lambda-1)h\ot fv_{\lambda}-e\ot f^2v_{\lambda}),$$
and
\begin{eqnarray*}
\lefteqn{
d^2(v_\lambda)=d(\lambda h\ot v_\lambda +2e\ot fv_\lambda)}\\
& = & \lambda h\ot (\lambda h\ot v_\lambda +2e\ot fv_\lambda)+2e\ot f(\lambda h\ot v_\lambda +2e\ot fv_\lambda).
\end{eqnarray*}
Thus,  
$$d_{\lambda}^2(v_\lambda)-\frac{\lambda^2}{\lambda +1}a_{\lambda +1}b_{\lambda +1}(v_\lambda)+\frac{(\lambda+2)^2}{\lambda +1}b_{\lambda -1} a_{\lambda -1}(v_\lambda)=0\quad (2\le \lambda \le p-3)$$
(for $\lambda =p-3$ the summand $ab(v_\lambda)$ is negligible, so it is sent to zero by the functor to $\Ver_p$).
\end{proof} 

\begin{remark}
The relations of our quiver $Q$ admit obvious deformations obtained
by choosing the coefficients of the monomials in \eqref{qrel1}-\eqref{qrel3} differently. It was shown to us by ChatGPT that for $p=7$ the resulting deformation of the algebra is not flat (see \S\ref{sec:Appendix}). However,
the deformation is flat if we keep relations \eqref{qrel1}-\eqref{qrel2} and deform
only relation \eqref{qrel3} (up to scaling of the generators this deformation is $1$-parameter). 
It would be interesting to investigate whether we have
similar behavior for $p>7$. \qed
\end{remark}

\begin{example}\label{quiver7}($p=7$)
Recall Example \ref{example7}, so that    
$\Ver_7^+=\langle \1, X, Y\rangle$, where $X:=L_3$ and $Y:=L_5$, and
$$X\ot X=\1\oplus X\oplus Y,\quad X\ot Y=X\oplus Y,\quad Y\ot Y=\1\oplus X.$$
By Theorem \ref{graded Cartan matrix}, the Cartan matrix of $\mathscr{Z}(\Ver_7^+)$ is given by 
\[
{\rm C}=\left(
\begin{array}{ccc}
3 & 2 & 1\\ 2 & 6 & 3 \\ 1 & 3 & 5
\end{array}
\right);
\]
that is, we have
$$[P(\1):\1]=3,\,\,[P(\1):X]=2,\,\,[P(\1):Y]=1,\,\,[P(X):\1]=2,\,\,[P(X):X]=6,$$
$$[P(X):Y]=3,\,\,[P(Y):\1]=1,\,\,[P(Y):X]=3,\,\,[P(Y):Y]=5.$$
 
By Theorem \ref{idealofrelations},   
the quiver $Q$ of $\mathscr{Z}(\Ver_7^+)$ is given by
\[
\xymatrix{
\bullet_{\1}\ar@/^.5pc/[r]^{b_1} & \bullet_X\ar@/^.5pc/[l]^{a_1}\ar@(ul,ur)^{d_2}\ar@/^.5pc/[r]^{b_3}&\bullet_Y\ar@/^.5pc/[l]^{a_3}\ar@(ul,ur)^{d_4}
},
\]
and 
$
I=
\left\langle
-a_1d_2,\,d_2a_3-a_3d_4,\,d_2b_1,\,
d_4b_3-b_3d_2,\,
d_2^2-\frac{4}{3}a_3b_3+\frac{16}{3}b_1a_1,\,
d_4^2+\frac{36}{5}b_3a_3
\right\rangle$.
Equivalently, after renormalization and using that $p=7$, we have
$$
I=
\left\langle
a_1d_2,d_2a_3-a_3d_4,\,d_2b_1,\,
d_4b_3-b_3d_2,\,
d_2^2+a_3b_3+3b_1a_1,\,
d_4^2+3b_3a_3
\right\rangle.$$
In particular, the element $d:=d_2+d_4$ is central.

Note that translating into (\ref{quiverp}) yields the following bound quiver $(Q,I)$:
\[
\xymatrix{
\bullet_{\1}\ar@/^.5pc/[r]^{b_1} & \bullet_X\ar@/^.5pc/[l]^{a_1}\ar@(ul,ur)^{d_1}\ar@/^.5pc/[r]^{b_2}&\bullet_Y\ar@/^.5pc/[l]^{a_2}\ar@(ul,ur)^{d_2}
},
\]
$$
I=
\left\langle
a_1d_1,d_1a_2-a_2d_2,\,d_1b_1,\,
d_2b_2-b_2d_1,\,
d_1^2+a_2b_2+3b_1a_1,\,
d_2^2+3b_2a_2
\right\rangle.$$
In particular, the element $d:=d_1+d_2$ is central. \qed
\end{example}

\section{The category $\mathscr{Z}(\Ver_5^+)$}\label{sec:The category 5}
In this section we set $\mathscr{Z}:=\mathscr{Z}(\Ver_5^+)$. Recall Example \ref{example5}, so that $X:=L_3$,  
$\Ver_5^+=\langle \1, X\rangle$, and $X\ot X=\1\oplus X$.
By Theorem \ref{graded Cartan matrix}, the Cartan matrix of $\mathscr{Z}$ is given by 
\[
{\rm C}=\left(
\begin{array}{cc}
2 & 1 \\ 1 & 3
\end{array}
\right);
\]
that is, we have
$$[P(\1):\1]=2,\,\,[P(\1):X]=1,\,\,[P(X):\1]=1,\,\,[P(X):X]=3.$$
The maximal ideal of $H={\rm Sym}(X)$ and its nonzero powers are 
$$\mathfrak{m}=X\oplus \1,\quad \mathfrak{m}^2=\1,$$
and we have
$$X\ot \mathfrak{m}=(X\ot X)\oplus X.$$
Then since the socle and cosocle of $H=P(\1)$ are $\1$, we have
\[
P(\1)=[\1\ X\ \1]^T,
\]  
and tensoring with the simple $X$ gives
\[
P(X)=[X\ (\1\oplus X)\ X]^T.
\]

\subsection{The bound quiver for $\mathscr{Z}(\Ver_5^+)$}\label{sec:The quiver 5} 
By Theorem \ref{idealofrelations},   
the quiver $Q$ of $\mathscr{Z}(\Ver_5^+)$ is given by
\[
\xymatrix{
\bullet_{\1}\ar@/^.5pc/[r]^{b} & \bullet_{X}\ar@/^.5pc/[l]^{a}\ar@(ul,ur)^{d}
}
\]
and the admissible ideal $I\subseteq J^2$ by $I=\langle ad,db,d^2+2ba\rangle$. In particular, the element $d$ is central.

\begin{corollary}\label{DB}
The category $\mathscr{Z}(\Ver_5^+)$ has exactly $10$ non-isomorphic indecomposable objects up to isomorphism.
\end{corollary}

\begin{proof} 
We have just seen that the indecomposable projective objects of the category
$\mathscr{Z}(\Ver_5^+)$ have the following properties: their heads and socles are isomorphic, and
the quotient of the radical by the socle is either simple or a direct sum of two simple objects.
It follows that the algebra $\mathbf{k}Q/I$ is a Brauer tree algebra, see e.g. \cite[Definition 5.10.4]{Z}.
The associated tree is the unique tree with two edges and one exceptional vertex of multiplicity $2$
at one of the ends. Hence, $\mathscr{Z}(\Ver_5^+)$ is of finite representation type; that is, it has only
finitely many indecomposable objects up to isomorphism. Moreover, the number of indecomposable 
objects can be computed using formula of Janusz \cite[Theorem 6.1]{Ja}, which says that
there are $e(em+1)$ indecomposable objects, where $e$ is the number of edges and $m$ is the
multiplicity of the exceptional vertex. Since in our case $e=m=2$, we get that the number of indecomposable
objects is $10$.
\end{proof}

\begin{remark}
It was pointed out to us by D. Benson that the principal block of the alternating group $A_5$
(in characteristic $5$)
has the same Brauer tree, so the category $\mathscr{Z}(\Ver_5^+)$ is Morita equivalent
to the principal block of $A_5$ as an abelian category. \qed
\end{remark}

\subsection{The indecomposables in $\mathscr{Z}(\Ver_5^+)$}\label{sec:The indecomposables}
Let ${\rm Gr}(\mathscr{Z})$ be the Grothendieck ring of $\mathscr{Z}$,
and set 
\begin{equation}\label{firstind5}
V_1:=\mathfrak{m}^*,\,\,V_2:=\mathfrak{m},\,\,W_1:=X\ot \mathfrak{m}^*,\,\,W_2:=X\ot \mathfrak{m}.
\end{equation}

As the socle/cosocle assignment is additive with respect to direct sums, we have the following lemma.

\begin{lemma}\label{verifies}
If $V$ is an object in $\mathscr{Z}$ for which ${\rm Soc}(V)$ or ${\rm Cosoc}(V)$ is simple, then $V$ is indecomposable.  Furthermore, if ${\rm Soc}(V)$ is simple then any subobject $0\ne V'\subset V$ is also indecomposable, and the dual statement holds when ${\rm Cosoc}(V)$ is simple and $V\twoheadrightarrow V'\ne 0$ is a quotient. \qed
\end{lemma}

\begin{lemma}\label{lem:XV}
$X\ot V_1\cong W_1$ and $X\ot V_2\cong W_2$
\end{lemma}

\begin{proof}
Tensoring the exact sequence $0\to \1\to P(\1)\to V_1\to 0$ with $X$ provides an exact sequence
\[
0\to X\to P(X)\to X\ot V_1\to 0,
\]
so that $X\ot V_1$ is the quotient of $P(X)$ by its socle, which is $W_1$.  By duality $X\ot V_2\cong (X\ot V_1)^\ast\cong W_1^\ast=W_2$.
\end{proof}

Let $S$ be a basis for ${\rm Ext}^1(X,X)$, or more precisely, let $S$ be the middle term of a non-split self-extension
\[
0\to X\to S\to X\to 0,
\]
representing a nonzero element of the one-dimensional space ${\rm Ext}^1(X,X)$. By Lemma \ref{verifies}, $S$ is indecomposable. Consider $S':=X\ot S$.  We calculate
\[
\dim\Hom(\1,S')=\dim\Hom(X,S)=1,
\]
\[
\dim\Hom(X,S')=\dim\Hom(\1\oplus X,S)=1,
\]
so that ${\rm Soc}(S')=\1\oplus X$.  Self-duality gives that ${\rm Cosoc}(S')=\1\oplus X$ as well.

\begin{lemma}\label{S'ind}
The object $S':=X\ot S$ is indecomposable.
\end{lemma}

\begin{proof}
Since the (co)socle of $S'$ is length $2$, we see that $S'$ decomposes into at most $2$ indecomposables.  Thus, since $[S']=2\1\oplus 2X$ in ${\rm Gr}(\mathscr{Z})$, if $S'$ is decomposable then 
\[
S'=S\oplus V',\ \text{or}\ S'=V_1\oplus V_2,
\]
where $V'$ is a non-split extension of $\1$.  Since by Theorem \ref{cohomology ring}(3), no such thing occurs, the first option cannot occur.  So the only option of decomposability is $S'=V_1\oplus V_2$.  We calculate
\[
\dim\Hom(S',V_1)=\dim\Hom(S,W_1)=1
\]
while $\dim\Hom(V_1\oplus V_2,V_1)=2$.  So the second option is also impossible, and we conclude that $S'$ is indecomposable.
\end{proof}

\begin{theorem}\label{complete list}
The following is a complete list of the indecomposables in $\mathscr{Z}(\Ver_5^+)$:
\begin{equation*}
\begin{array}{c}
{\1},\ \ X,\ \  P(\1)={[\1\ X\ \1]}^T,\ \  P(X)={[X\ \1\oplus X\ X]}^T,\\
V_1={[\1\ X]}^T,\ \  V_2=V_1^*= {[X\ \1]}^T,\ \ W_1={[X\ \1\oplus X]}^T,\ \ W_2=W_1^*={[\1\oplus X\ X]}^T,\\
S={[X\ X]}^T,\quad and \quad S'={[\1\oplus X\ \1\oplus X]}^T.
\end{array}
\end{equation*}
\end{theorem}

\begin{proof}
By Lemmas \ref{verifies} and \ref{S'ind}, these are indecomposables, and by Corollary \ref{DB}, these are all the indecomposables (up to isomorphism).
\end{proof}

\subsection{The Green ring for $\mathscr{Z}(\Ver_5^+)$}\label{sec:The Green ring}

Retain the notation of \S\ref{sec:The indecomposables}.

\subsubsection{The projectives}\label{sec:The projectives}

\begin{lemma}\label{lem:VP}
If $[V]=n\1\oplus mX$ in ${\rm Gr}(\mathscr{Z})$, then
\[
V\ot P(\1)=nP(\1)\oplus mP(X)\ \ \text{and}\ \ V\ot P(X)=mP(\1)\oplus (m+n)P(X).
\]
\end{lemma}

\begin{proof}
Induction on the length of $V$.
\end{proof}

\subsubsection{Multiplication by $X$}\label{sec:Multiplication by X}

\begin{lemma}
In the Green ring we have
\[
X\ot X=X\oplus\1,\ X\ot V_1=W_1,\ X\ot V_2=W_2, \ X\ot S=S',
\]
\[
X\ot W_1=W_1\oplus V_1,\ X\ot W_2=W_2\oplus V_2,\ \text{and}\ X\ot S'=S'\oplus S.
\]
\end{lemma}

\begin{proof}
The first centered line of formulas follows from what we have already calculated. As for the second one, 
suppose $V'$ is such that $V'\cong X\ot V$.  Then
\[
X\ot V'\cong (X\ot X)\ot V\cong (X\oplus \1)\ot V=(X\ot V)\oplus V\cong V'\oplus V. 
\]
From this generic calculation, and the expressions $W_i=X\ot V_i$ and $S'=X\ot S$, we obtain the second centered line of prescribed formulas.
\end{proof}

\subsubsection{Multiplication by $V_1$ and $V_2$}\label{sec:Multiplication by V1 and V2}

\begin{lemma}
$V_1\ot V_2=\1\oplus P(X)$.
\end{lemma}

\begin{proof}
Since $\dim\Hom(\1,V_1\ot V_2)=1$ and 
\[
\dim\Hom(X,V_1\ot V_2)=\dim\Hom(W_1,V_1)=1,
\]
the socle of $V_1\ot V_2$ is $\1\oplus X$.  Since this object is self-dual, the cosocle is $\1\oplus X$ as well. Since  $[V_1\ot V_2]=2\1\oplus 3X$ in ${\rm Gr}(\mathscr{Z})$, the product is length $5$.  These length and socle constraints demand that $V_1\ot V_2$ has exactly two indecomposable summands
\[
V_1\ot V_2=V\oplus V'
\]
with both the socle and cosocle of $V$ and $V'$ of length $1$.  It follows that either $V=P(\1)$ and $V'=S$, or $V=P(X)$ and $V'=\1$.
But since $\dim(V_1),\dim(V_2)\ne 0$, $\1$ appears as a summand in $V_1\ot V_2$, so we get the result.
\end{proof}

\begin{lemma}
$V_1\ot V_1=V_2\ot V_2=P(\1)\oplus S$.
\end{lemma}

\begin{proof}
We calculate
\[
\dim\Hom(\1,V_1\ot V_1)=\dim\Hom(V_2,V_1)=1,
\]
\[\ \dim\Hom(X,V_1\ot V_1)=\dim\Hom(W_2,V_1)=1,
\]
\[
\dim\Hom(V_1\ot V_1,\1)=1,\,\, and\,\, \dim\Hom(V_1\ot V_1,X)=\dim\Hom(V_1,W_2)=1.
\]
Therefore the socle and cosocle of $V_1\ot V_1$ are both $\1\oplus X$.  Also $[V_1\ot V_1]=2\1\oplus 3X$ in ${\rm Gr}(\mathscr{Z})$, so $V_1\ot V_1$ has at most $2$ indecomposables in its decomposition, and by considering the length of the indecomposables with socle $\1\oplus X$, we see that $V_1\ot V_1$ is decomposable.  Hence
\[
V_1\ot V_1=V\oplus V'
\]
with $V$ having socle $\1$ and $V'$ having socle $X$, and both the socles and cosocles of $V$ and $V'$ must be length $1$.  Since $W_1$, $W_2$ and $S'$ have too large socle or cosocle to appear in this decomposition, and the fact that $V_1\ot V_1$ is length $5$, it follows that $V=P(\1)$ and $V'=S$ or $V=P(X)$ and $V'=\1$.
But we have
\[
\dim\Hom(V_1,V_1\ot V_1)=\dim\Hom(\1\oplus P(X),V_1)=1
\]
while $\dim\Hom(V_1,\1\oplus P(X))=2$, so that the second option is eliminated.  Thus $V_1\ot V_1=P(\1)\oplus S$.  Since $P(\1)\oplus S$ is self-dual, we find that 
$$V_2\ot V_2=(V_1\ot V_1)^\ast=P(\1)\oplus S$$
as well.
\end{proof}

\begin{lemma}
$V_1\ot W_1=V_2\ot W_2=P(X)\oplus S'$, and 
$$V_1\ot W_2=V_2\ot W_1=P(\1)\oplus P(X)\oplus X.$$
\end{lemma}

\begin{proof}
This follows from the previous two lemmas and the fact that 
$$V_i\ot W_j=X\ot(V_i\ot V_j)$$
for every $1\le i,j\le 2$.
\end{proof}

\begin{lemma}
$V_1\ot S=P(X)\oplus V_2$, and $V_2\ot S=P(X)\oplus V_1$.
\end{lemma}

\begin{proof}
We have
\[
V_1^{\ot 3}=V_1\ot V_2\ot V_2=(\1\oplus P(X))\ot V_2=P(\1)\oplus 2P(X)\oplus V_2,
\]
and also
\[
V_1^{\ot 3}=V_1\ot (P(\1)\oplus S)=P(\1)\oplus P(X)\oplus V_1\ot S.
\]
Thus, $V_1\ot S=P(X)\oplus V_2$, and taking the dual gives $V_2\ot S=P(X)\oplus V_1$.
\end{proof}

\begin{corollary}
$V_1\ot S'=P(\1)\oplus P(X)\oplus W_2$ and $V_2\ot S'=P(\1)\oplus P(X)\oplus W_1$.
\end{corollary}

\begin{proof}
Considering $S'=X\ot S$ and applying the braiding gives the result.
\end{proof}

\subsubsection{Multiplication by $S$}\label{sec:Multiplication by S}

\begin{lemma}
$S\ot S=\1\oplus P(\1)\oplus P(X)$.
\end{lemma}

\begin{proof}
Since
$
\dim\Hom(\1,S\ot S)=2$ and $\dim\Hom(X,S\ot S)=1$,
the socle of $S\ot S$ is $2\1\oplus X$.  Self-duality provides the cosocle as $2\1\oplus X$ as well.  Then $[S\ot S]=4\1\oplus 4X$ in ${\rm Gr}(\mathscr{Z})$.  So we have a decomposition into three indecomposables, necessarily of total length 8.  The socle and cosocle of each indecomposable in this decomposition must be simple, so that our options are
\[
\1\oplus P(\1)\oplus P(X)\quad or \quad P(\1)\oplus P(\1)\oplus S.
\]
Since $\dim(S)\ne 0$, $\1$ must appear as a summand in the indecomposable decomposition of $S\ot S$, which implies the claim.
\end{proof}

\subsubsection{A complete presentation}\label{sec:prod}

We omit the multiplication with the projectives, as they are clear from Lemma \ref{lem:VP}. The following hold:
\[
X\ot-\left\{\begin{array}{l}
X= \1\oplus X\\
V_1= W_1\\
V_2= W_2\\
W_1= W_1\oplus V_1\\
W_2= W_2\oplus V_2\\
S= S'\\
S'= S'\oplus S
\end{array}\right.
\]
\[
\begin{array}{ll}
V_1\ot-\left\{\begin{array}{l}
V_1= P(\1)\oplus S\\
V_2= \1\oplus P(X)\\
W_1= P(X)\oplus S'\\
W_2= X\oplus P(\1)\oplus P(X)\\
S= P(X)\oplus V_2\\
S'= P(\1)\oplus P(X)\oplus W_2\\
\end{array}\right.
& \qquad
V_2\ot-\left\{\begin{array}{l}
V_2= P(\1)\oplus S\\
W_1= X\oplus P(\1)\oplus P(X)\\
W_2= P(X)\oplus S'\\
S= P(X)\oplus V_1\\
S'= P(\1)\oplus P(X)\oplus W_1
\end{array}\right.
\end{array}
\]
\[
\begin{array}{ll}
W_1\ot-\left\{\begin{array}{l}
W_1= P(\1)\oplus P(X)\oplus S'\oplus S\\
W_2= \1\oplus X\oplus P(\1)\oplus 2P(X)\\
S= P(\1)\oplus P(X)\oplus W_2\\
S'= P(\1)\oplus 2P(X)\oplus W_2\oplus V_2
\end{array}\right.
&
W_2\ot-\left\{\begin{array}{l}
W_2= P(\1)\oplus P(X)\oplus S'\oplus S\\
S= P(\1)\oplus P(X)\oplus W_1\\
S'= P(\1)\oplus 2P(X)\oplus W_1\oplus V_1
\end{array}\right.
\end{array}
\]
\[
S\ot S= \1\oplus P(\1)\oplus P(X),\,\,
S\ot S'= X\oplus P(\1)\oplus 2P(X),\quad and
\]
\[
S'\ot S'= \1 \oplus X\oplus 2P(\1)\oplus 3P(X).
\]

\subsection{The cohomology ring}\label{sec:The cohomology ring 5}
Let $\mathscr{Z}:=\mathscr{Z}(\Ver_5^+)$, set $d:={\rm FPdim}(X)$, and let $P^{\bullet}$ be as in (\ref{minprojres}).

Let us use the above results to see directly that 
there is an isomorphism of graded commutative algebras 
$$H^{\bullet}(\mathscr{Z})\cong \mathbf{k}[x,y]/\langle y^2\rangle,$$
where $x$ has degree $4$ and $y$ has degree $3$, as given by Theorem \ref{cohomology ring}.

We have
$$\bullet \qquad \ker(\partial^{0})=V_2,\quad P^1=P(V_2)=P(X),\quad and\quad \partial^1:P(X)\twoheadrightarrow V_2$$ 
is a lifting of $q:P(X)\twoheadrightarrow X$ along $V_2\twoheadrightarrow X$. Thus, $\ker (\partial^1)\subseteq \ker (q)=W_2$. Since ${\rm FPdim}(V_2)=1+d$, ${\rm FPdim}(W_2)=1+2d$, and  
$${\rm FPdim}(X)+{\rm FPdim}(W_2)={\rm FPdim}(P(X))={\rm FPdim}(V_2)+{\rm FPdim}(\ker (\partial^1)),$$
we get that ${\rm FPdim}(\ker (\partial^1))=2d$. Hence, by Theorem \ref{complete list}, $\ker (\partial^1)=S$, so
$$\bullet \qquad \ker (\partial^1)=S,\quad
P^2=P(S)=P(X),\quad and\quad 
\partial^2:P(X)\twoheadrightarrow S$$
is a lifting of $q:P(X)\twoheadrightarrow X$ along $S\twoheadrightarrow X$. Thus, $\ker (\partial^2)\subseteq \ker (q)=W_2$. Since
$$1+3d={\rm FPdim}(P(X))={\rm FPdim}(S)+{\rm FPdim}(\ker (\partial^2))=2d+{\rm FPdim}(\ker (\partial^2)),$$
we get that ${\rm FPdim}(\ker (\partial^2))=1+d$. Hence, by Theorem \ref{complete list}, $\ker (\partial^2)=V_1$, so
$$\bullet \qquad \ker (\partial^2)=V_1,\quad
P^3=P(V_1)=P(\1),\quad and\quad 
\partial^3:P(\1)\twoheadrightarrow V_1$$
is a lifting of $\varepsilon:P(\1)\twoheadrightarrow \1$ along $V_1\twoheadrightarrow \1$, so $\ker (\partial^3)\subseteq \ker (\varepsilon)=V_2$. Since 
$$2+d={\rm FPdim}(P(\1))={\rm FPdim}(V_1)+{\rm FPdim}(\ker (\partial^3))=1+d+{\rm FPdim}(\ker (\partial^3)),$$
we get that ${\rm FPdim}(\ker (\partial^3))=1$.
Thus, we have
$$\bullet \qquad \ker (\partial^3)=\1,\quad
P^4=P^0=P(\1),\quad and\quad \partial^4=\partial^0=\varepsilon:P(\1)\twoheadrightarrow \1.$$
Therefore, we see that the resolution $P^{\bullet}$ is periodic with period $4$, $\Hom(P^1,\1)=0$, $\Hom(P^2,\1)=0$, $\Hom(P^3,\1)=\mathbf{k}$ and $\Hom(P^4,\1)=\mathbf{k}$, as claimed in Theorem \ref{cohomology ring}.

Consider next the graded algebra ${\rm Ext}^{\bullet}(X,X)$. We have for each $n\ge 0$, 
\[
{\rm Ext}^{n}(X,X)=\Hom(P^n\ot X,X)=\mathbf{k}.
\]
We know that ${\rm Ext}^{1}(X,X)={\rm sp}\{S\}$. 

\begin{question}\label{extxx}
What is the algebra structure of ${\rm Ext}^{\bullet}(X,X)$? In particular, is the Yoneda algebra ${\rm Ext}^{\bullet}(X,X)$ generated in degree $1$? \qed 
\end{question}

\subsection{The semisimplification of $\mathscr{Z}(\Ver_5^+)$}\label{sec:The semisimplification 5}
Let $\overline{\mathscr{Z}}:=\overline{\mathscr{Z}(\Ver_5^+)}$. Since 
$$\dim(\1)=1\quad and \quad \dim(X)=3,$$
it follows from Theorem \ref{complete list} that $P(\1)$ and $P(X)$ are the only indecomposables of dimension $0$, so the braided fusion category $\overline{\mathscr{Z}}$ has exactly $8$ simple objects
$$\1,\,X,\,V_1,\,V_2=V_1^*,\,W_1,\,W_2=W_1^*,\,S=S^*,\,S'=(S')^*,$$
with categorical dimensions
$$\dim(V_1)=\dim(V_2)=-1,\,\dim(S)=1,\,\dim(W_1)=\dim(W_2)=2,\,\dim(S')=3,$$
and by \S\ref{sec:prod}, they satisfy the following fusion rules:
\[
\begin{array}{ll}
X\ot-\left\{\begin{array}{l}
X= \1\oplus X\\
V_1= W_1\\
V_2= W_2\\
W_1= W_1\oplus V_1\\
W_2= W_2\oplus V_2\\
S= S'\\
S'= S'\oplus S
\end{array}\right.
& \qquad
V_1\ot-\left\{\begin{array}{l}
V_1= S\\
V_2= \1\\
W_1= S'\\
W_2=X\\
S= V_2\\
S'= W_2\\
\end{array}\right.
\end{array}
\]
\[
\begin{array}{ll}
V_2\ot-\left\{\begin{array}{l}
V_2= S\\
W_1= X\\
W_2=S'\\
S= V_1\\
S'= W_1
\end{array}\right.
& \qquad \quad
W_1\ot-\left\{\begin{array}{l}
W_1=S'\oplus S\\
W_2= \1\oplus X\\
S= W_2\\
S'=W_2\oplus V_2
\end{array}\right.
\end{array}
\]
\[
W_2\ot W_2=S'\oplus S,\,\, W_2\ot S= W_1,\,\, W_2\ot S'= W_1\oplus V_1,
\]
\[
S\ot S=\1,\,\, S\ot S'=X,\,\, and \,\, 
S'\ot S'=\1\oplus X.
\]
 
In particular, the invertible objects of $\overline{\mathscr{\mathscr{Z}}}$ are $\1,S,V_1,V_2$, and they form a pointed fusion subcategory $\mathscr{E}:=\langle \1,S,V_1,V_2\rangle \subseteq \overline{\mathscr{Z}}$. Also, $\Ver_5^+=\langle \1,X\rangle \subseteq \overline{\mathscr{Z}}$ is a symmetric fusion subcategory.

\begin{theorem}\label{semisimp}
The following hold:
\begin{enumerate}
\item
$\mathscr{E}$ and $\Ver_5^+$ centralize each other.
\item
The braided pointed fusion subcategory $\mathscr{E}\subseteq \overline{\mathscr{Z}}$ is equivalent to $\Rep(\mathbb{Z}/4\mathbb{Z},z)$ as symmetric categories, where $z$ is the unique element of order two.
\item
$\overline{\mathscr{Z}}$ is symmetric, and 
$\overline{\mathscr{Z}}\simeq \mathscr{E} \boxtimes \Ver_5^+$ as symmetric fusion categories.  
\end{enumerate}
\end{theorem}

\begin{proof}
(1) Since $\mathscr{E}_{{\rm ad}}={\rm Vec}$, $(\mathscr{E}_{{\rm ad}})'=\overline{\mathscr{Z}}$, and since by \cite[Proposition 8.22.6]{EGNO}, $(\mathscr{E}_{{\rm ad}})'=(\mathscr{E}')^{{\rm co}}$, it follows that $(\mathscr{E}')^{{\rm co}}=\overline{\mathscr{Z}}$. In particular, $X\in (\mathscr{E}')^{{\rm co}}$; that is, $X\ot X^*=\1\oplus X\in \mathscr{E}'$, so $X\in \mathscr{E}'$. Thus, $\mathscr{E}$ and $\Ver_5^+$ centralize each other.

(2) By the results in \cite[Section 8.4]{EGNO}, it suffices to show that $\mathscr{E}$ is symmetric. To this end, let $A,B$ be two simples in $\mathscr{E}$. Since $A\ot B$ is simple, $\bar{{\rm c}}^2_{A,B}=\lambda\id_{A\ot B}$ for some $\lambda\in \mathbf{k}^{\times}$. Since by Theorem \ref{main result}(4), $\id - {\rm c}^2_{A,B}$ is nilpotent, it follows that so is $\id - \bar{{\rm c}}^2_{A,B}$, which implies that $\lambda=1$.

(3) The fusion rules of $\overline{\mathscr{Z}}$ imply that $\mathscr{E}$ and $\Ver_5^+$ determine an exact factorization of $\overline{\mathscr{Z}}$ in the sense of \cite{G1}. Since $\overline{\mathscr{Z}}$ is braided, \cite[Corollary 3.9]{G1} implies that $\overline{\mathscr{Z}}\simeq \mathscr{E} \boxtimes \Ver_5^+$ as fusion categories. Thus, the claim follows from (1) and (2) and the fact that $\dim(V_1)=-1$.
\end{proof}

\subsection{Minimal extensions}\label{sec:Minimal extensions 5}
Let $\mathscr{D}$ be a finite {non-degenerate} braided tensor category over $\mathbf{k}$ that contains $\Ver_5^+$ as a Lagrangian subcategory, so that by \cite{S}, we have ${\rm FPdim}(\mathscr{D})={\rm FPdim}(\Ver_5^+)^2$.
It is not hard to verify that $\Ver_5^+$ must be the semisimple part of $\mathscr{D}$, and consequently that $\mathscr{D}$ and $\mathscr{Z}(\Ver_5^+)$ have the same Grothendieck and Green rings. 

\begin{question}\label{minext5}
Does the above imply that $\mathscr{D}\simeq \mathscr{Z}(\Ver_5^+)$ as braided tensor categories? In other words, is it true that both the associativity and braiding of $\mathscr{Z}(\Ver_5^+)$ cannot be non-trivially deformed? \qed
\end{question}

\section{Appendix by ChatGPT under guidance from Pavel Etingof}\label{sec:Appendix}

Throughout, $\mathbf{k}$ is an algebraically closed field of characteristic $7$.  Label the
three vertices of the quiver by $0,1,2$, corresponding respectively to
$\mathbf 1,X,Y$.  I use the usual path-algebra convention that in a product
$pq$ the path $q$ is traversed first.  Thus
\[
 b_1:0\longrightarrow 1,\qquad a_1:1\longrightarrow 0,
 \qquad b_2:1\longrightarrow 2,\qquad a_2:2\longrightarrow 1,
\]
and $d_1,d_2$ are loops at $1,2$, respectively.

The algebra in Example \ref{quiver7} is
\begin{equation}\label{eq:original}
 \mathbf{k} Q\Big/\left\langle
 a_1d_1,\ d_1a_2-a_2d_2,\ d_1b_1,\ d_2b_2-b_2d_1,
 \ d_1^2+a_2b_2+3b_1a_1,\ d_2^2+3b_2a_2
 \right\rangle .
\end{equation}

\subsection{Parameters modulo renormalization of the arrows}
After dividing each relation by one of its nonzero coefficients, the most general
coefficient deformation in a formal neighborhood of \eqref{eq:original} is
\begin{equation}\label{eq:general}
 A(\alpha,\beta,\gamma,\delta,\varepsilon)
 :=\mathbf{k} Q\Big/\left\langle
 \begin{aligned}
 &a_1d_1,\qquad d_1a_2-\alpha a_2d_2,\qquad d_1b_1,\qquad
 d_2b_2-\beta b_2d_1,\\
 &d_1^2+\gamma a_2b_2+\delta b_1a_1,
 \qquad d_2^2+\varepsilon b_2a_2
 \end{aligned}
 \right\rangle ,
\end{equation}
where all five parameters are nonzero.

Renormalize the arrows by
\[
 d_1'=s d_1,\quad d_2'=t d_2,\quad
 a_2'=u a_2,\quad b_2'=v b_2,\quad
 a_1'=x a_1,\quad b_1'=y b_1.
\]
Then the parameters transform as
\begin{equation}\label{eq:scaling}
 \alpha'=\alpha\frac{s}{t},\qquad
 \beta'=\beta\frac{t}{s},\qquad
 \gamma'=\gamma\frac{s^2}{uv},\qquad
 \delta'=\delta\frac{s^2}{xy},\qquad
 \varepsilon'=\varepsilon\frac{t^2}{uv}.
\end{equation}
Consequently the two quantities
\begin{equation}\label{eq:invariants}
 h:=\alpha\beta,
 \qquad
 \lambda:=\frac{\alpha^2\varepsilon}{\gamma}
\end{equation}
are invariant under renormalization.  Conversely, they are a complete set of
invariants for the action of arrow rescalings: taking $t=\alpha s$,
$uv=\gamma s^2$, and $xy=\delta s^2$ reduces \eqref{eq:general} to
\begin{equation}\label{eq:canonical}
 A_{h,\lambda}:=\mathbf{k} Q\Big/\left\langle
 \begin{aligned}
 &a_1d_1,\qquad d_1a_2-a_2d_2,\qquad d_1b_1,\qquad
 d_2b_2-hb_2d_1,\\
 &d_1^2+a_2b_2+b_1a_1,
 \qquad d_2^2+\lambda b_2a_2
 \end{aligned}
 \right\rangle .
\end{equation}
The original algebra has
\[
 (h,\lambda)=(1,3).
\]
Indeed, the coefficient $3$ of $b_1a_1$ in \eqref{eq:original} is removed by
renormalizing, for example, $a_1' =3a_1$.

\subsection{The flat locus}

\begin{proposition}\label{prop:classification}
Assume that $h,\lambda\in\mathbf{k}^\times$.
\begin{enumerate}
 \item If $h=1$ and $\lambda\ne1$, then $A_{1,\lambda}$ has dimension $26$,
 is concentrated in degrees $0,\ldots,4$, and has Hilbert series
 \[
  3+6z+8z^2+6z^3+3z^4.
 \]
 Its Cartan matrix, with the vertices ordered as $0,1,2$, is
 \[
  \begin{pmatrix}
   3&2&1\\
   2&6&3\\
   1&3&5
  \end{pmatrix}.
 \]
 Moreover, these algebras form a finite free family over
 \[
  R:=\mathbf{k}[\lambda,\lambda^{-1},(\lambda-1)^{-1}].
 \]

 \item If $h\ne1$, then $A_{h,\lambda}$ has dimension $22$ and Cartan matrix
 \[
  \begin{pmatrix}
   3&2&1\\
   2&4&3\\
   1&3&3
  \end{pmatrix}.
 \]
 Thus varying $h$ away from $1$ is not a flat deformation of the original
 algebra.

 \item If $h=1$ and $\lambda=1$, then $A_{1,1}$ is infinite-dimensional.
\end{enumerate}
\end{proposition}

\begin{proof}
We first treat $h=1$.  For the Diamond-lemma computation only, write a path as a
word in the order in which its arrows are traversed, and put
\[
 u=b_1,\qquad v=a_1,\qquad x=b_2,\qquad y=a_2,
 \qquad s=d_1,\qquad t=d_2.
\]
Use the degree-lexicographic order induced by
\[
 v>y>x>s>t>u.
\]
Over $R$, the six quadratic relations in \eqref{eq:canonical} have the following
oriented form:
\begin{equation}\label{eq:quadratic-rules}
 us\mapsto0,\quad sv\mapsto0,\quad
 ys\mapsto ty,\quad xt\mapsto sx,
 \quad vu\mapsto-ss-xy,\quad
 yx\mapsto-\lambda^{-1}tt.
\end{equation}
Completing the reduction system adds
\begin{equation}\label{eq:higher-rules}
 sxy\mapsto-sss,\qquad tyv\mapsto0,
 \qquad sssx\mapsto0,\qquad ttty\mapsto0,
 \qquad s^5\mapsto0,\qquad t^5\mapsto0.
\end{equation}
For clarity, the nontrivial ambiguity calculations producing these additional
rules are, up to multiplication by units,
\[
\begin{array}{c|c}
 \text{ambiguous word}&\text{new consequence}\\ \hline
 svu\quad(\text{equivalently }vus)&sxy+sss=0\\
 ysv&tyv=0\\
 sxyx&(\lambda-1)sssx=0\\
 tyvu&(\lambda-1)ttty=0\\
 sssxy&s^5=0\\
 tttyx&t^5=0.
\end{array}
\]
Here one uses that both $\lambda$ and $\lambda-1$ are units.  A direct check of
the remaining overlaps shows that they all resolve.  Hence the Diamond lemma
implies that the standard words form an $R$-basis.

Returning to the usual path notation, this basis is
\begin{align*}
\mathcal B={}&
 e_0,e_1,e_2;\\
&b_1,a_1,d_1,b_2,a_2,d_2;\\
&a_1b_1,\ b_2b_1,\ a_2b_2,\ d_1^2,\ b_2d_1,
 \ a_1a_2,\ a_2d_2,\ d_2^2;\\
&a_2b_2b_1,\ a_1a_2b_2,\ d_1^3,\ b_2d_1^2,
 \ a_2d_2^2,\ d_2^3;\\
&a_1a_2b_2b_1,\ d_1^4,\ d_2^4.
\end{align*}
Thus the family is free of rank $3+6+8+6+3=26$ over $R$.  Grouping the basis
paths by source and target gives
\[
\begin{array}{c|c|c|c}
 \text{source}\backslash\text{target}&0&1&2\\ \hline
0& e_0,\ a_1b_1,\ a_1a_2b_2b_1
 &b_1,\ a_2b_2b_1
 &b_2b_1\\[2pt]
1&a_1,\ a_1a_2b_2
 &e_1,\ d_1,\ a_2b_2,\ d_1^2,\ d_1^3,\ d_1^4
 &b_2,\ b_2d_1,\ b_2d_1^2\\[2pt]
2&a_1a_2
 &a_2,\ a_2d_2,\ a_2d_2^2
 &e_2,\ d_2,\ d_2^2,\ d_2^3,\ d_2^4.
\end{array}
\]
This proves the first assertion and the stated Cartan matrix.

Now suppose $h\ne1$.  In traversal notation, the quadratic rule
$xt\mapsto sx$ is replaced by
\[
 xt\mapsto h sx.
\]
The two ambiguity calculations at the words $svu$ and $vus$ give, respectively,
\[
 sss+sxy=0,
 \qquad
 sss+h sxy=0.
\]
Since $h-1$ is invertible, this forces
\[
 sxy=sss=0.
\]
Likewise, the ambiguity at $yxt$ gives $(h-1)ttt=0$, hence $ttt=0$, while $ysv$
gives $tyv=0$.  The resulting confluent reduction system has as standard words
exactly
\[
 \mathcal B\setminus\{d_1^3,d_2^3,d_1^4,d_2^4\}.
\]
This is a basis of cardinality $22$, and grouping it by endpoints gives the
second Cartan matrix.

Finally, take $h=\lambda=1$.  There is a quotient representation in
$M_2(\mathbf{k}[z])$ given by
\[
 e_0\mapsto0,\quad e_1\mapsto E_{11},\quad e_2\mapsto E_{22},
\]
\[
 d_1\mapsto zE_{11},\qquad d_2\mapsto zE_{22},\qquad
 b_2\mapsto zE_{21},\qquad a_2\mapsto-zE_{12},
 \qquad a_1,b_1\mapsto0.
\]
All defining relations are satisfied, and the image contains $z^nE_{11}$ for
all $n\ge0$.  Hence $A_{1,1}$ is infinite-dimensional.
\end{proof}

\begin{corollary}\label{cor:answer}
In a formal neighborhood of the algebra in Example 3.12, the locus of flat
coefficient deformations is
\[
 h=\alpha\beta=1.
\]
Modulo renormalization of the generators, it is one-dimensional, with parameter
\[
 \lambda=\frac{\alpha^2\varepsilon}{\gamma},
 \qquad \lambda(0)=3.
\]
Equivalently, a representative of the universal one-parameter coefficient
deformation is
\begin{equation}\label{eq:formal-family}
 \mathbf{k}[[\tau]]Q\Big/\left\langle
 \begin{aligned}
 &a_1d_1,\quad d_1a_2-a_2d_2,\quad d_1b_1,
 \quad d_2b_2-b_2d_1,\\
 &d_1^2+a_2b_2+3b_1a_1,
 \quad d_2^2+(3+\tau)b_2a_2
 \end{aligned}
 \right\rangle .
\end{equation}
It is free of rank $26$ over $\mathbf{k}[[\tau]]$ and has the same Cartan matrix as the
original algebra.  The parameter $\tau$ cannot be removed by rescaling the arrows.
\end{corollary}

\begin{proof}
At the original point, both $\lambda=3$ and $\lambda-1=2$ are units, so every
formal variation of $\lambda$ remains in the open set of Proposition
\ref{prop:classification}(1).  The overlap computation used in part (2) gives the relation
\[
 (h-1)sxy=0
\]
over any coefficient base.  In the original fiber, $sxy=-s^3$ is nonzero.
If the deformation algebra is flat over a local base, it is free, and an
element whose reduction is nonzero has zero annihilator.  Hence $h-1=0$.
After imposing $h=1$, the only remaining renormalization invariant is
$\lambda$.  Formula \eqref{eq:formal-family} is the specialization
$\lambda=3+\tau$.
\end{proof}

\medskip
\noindent
\textbf{Remark.}
For $\lambda\in\mathbf{k}^\times\setminus\{1\}$, the parameter $\lambda$ in fact
also distinguishes the resulting algebras up to arbitrary algebra isomorphism,
not merely up to diagonal renormalization of the arrows.  Indeed, the positive
path-length part is the Jacobson radical, so the path-length grading is the
radical grading.  Any algebra isomorphism therefore induces a graded
isomorphism.  The diagonal Cartan entries $3,6,5$ force such an isomorphism to
fix the three vertices, and every degree-one arrow space is one-dimensional.
Thus the induced graded isomorphism acts by arrow rescalings, under which
$\lambda$ is invariant.

For completeness, over the dual numbers $\mathbf{k}[\eta]/(\eta^2)$ write
\[
 \alpha=1+A\eta,\quad \beta=1+B\eta,\quad
 \gamma=1+G\eta,\quad \delta=3+D\eta,\quad
 \varepsilon=3+E\eta.
\]
Then flatness is equivalent to
\[
 A+B=0,
\]
and, modulo infinitesimal renormalization, the unique deformation parameter is
\[
 \lambda-3=(E+6A-3G)\eta.
\]
The coefficient $D$ is pure gauge: it is removed by rescaling $a_1$ and $b_1$.

\medskip
\noindent
\textbf{Conclusion.}
Changing the coefficients does give a nontrivial flat deformation, but not with
all coefficients varying independently.  After quotienting by renormalization,
there is exactly one coefficient modulus near the original point.  It may be
taken to be the coefficient of $b_2a_2$ in the last relation, while the product
of the two coefficients in the transport relations must remain equal to $1$. \qed


\begin{thebibliography}{EGNO}
\bibitem
[BGS]{BGS}
A.~Beilinson, V.~Ginzburg, W.~S\"orgel. Koszul duality patterns in representation theory.
{\em J. of Amer. Math. Soc.} {\bf 16} (1996), no. 2, p. 473-527.

\bibitem
[CEO]{CEO} 
K.~Coulembier, P.~Etingof, V.~Ostrik. On Frobenius exact symmetric tensor categories.
{\em Annals of Mathematics} {\bf 197} (2023), no. 3, p. 1235-1279. 

\bibitem
[D]{D} 
P. Deligne. Categories Tannakiennes. In ``The Grothendieck Festschrift, Vol.II,” {\em Progress in Mathematics} {\bf 87} (1990), 111–195.

\bibitem
[DF]{DF}
P. Donovan and M.R. Freislich. The Representation Theory of Finite Graphs and Associated
Algebras. {\em Carleton Mathematical Lecture Notes} {\bf No. 5}. Carleton University, Ottawa, ON, 1973.
iii+83 pp.

\bibitem
[DR]{DR}
V. Dlab and C.M. Ringel. Indecomposable representations of graphs and algebras. {\em Mem. Amer.
Math. Soc.} {\bf 6} (1976), no. 173.

\bibitem
[E]{E} 
P. Etingof. Koszul duality and the PBW theorem in symmetric tensor categories in positive characteristic. {Advances in Mathematics} {\bf 327} (2018), 128–160.

\bibitem
[EG]{EG} 
P. Etingof and S. Gelaki. Quasisymmetric and unipotent tensor categories. 
{\em Mathematical Research Letters} {\bf 15} (2008), no. 5, 857–866.

\bibitem
[ES]{ES} 
P. Etingof and A. Snowden. The Drinfeld center of an oligomorphic tensor category. {\em arXiv:2604.00290}.

\bibitem
[EOV]{EOV} 
P. Etingof, V. Ostrik, and S. Venkatesh. Computations in symmetric fusion categories in characteristic $p$. 
{\em International Mathematics Research Notices} (2017) (2), 468–489.

\bibitem
[EGNO]{EGNO} 
P. Etingof, S. Gelaki, D. Nikshych, and V. Ostrik. Tensor categories. {\em AMS Mathematical Surveys and Monographs book series} {\bf 205} (2015), 362 pp.

\bibitem 
[G]{G1} 
S. Gelaki. Exact factorizations and extensions of fusion categories. {\em Journal of Algebra} {\bf 480} (2017), 505–518.

\bibitem[GK]
{GK} S.~Gelfand and  D.~Kazhdan.
Examples of tensor categories.
{\em Inventiones Mathematicae} \textbf{109} (1992), no. 3, 595–617.

\bibitem
[GM]{GM} G.~Georgiev and O.~Mathieu. 
Fusion rings for modular representations of Chevalley groups.
{\em Contemporary Mathematics} \textbf{175} (1994), 89–100.

\bibitem
[J]{Ja} G.~J.~Janusz. Indecomposable modules for finite groups. {\em Ann. of Math.} 
(2) 89 (1969), 209-241.

\bibitem
[LZ]{LZ} Z. Liu and S. Zhu. Centers of braided tensor categories. {\em J. Algebra} {\bf 614} (2023), 115–153.

\bibitem
[MOV]{MOV} A.~Malkin, V.~Ostrik, M.~Vybornov. Quiver varieties and Lusztig’s algebra. {\em Advances in Mathematics} {\bf 203} (2006) 514-536.

\bibitem 
[O]{O} 
V. Ostrik. On symmetric fusion categories in positive characteristic. {\em Selecta Mathematica} {\bf 26} (3), 36 (2020).
		
\bibitem 
[S]{S} 
K. Shimizu. The monoidal center and the character algebra. {\em Journal of Pure and Applied Algebra} 
{\bf 221} (2017) 2338–2371.

\bibitem 
[V1]{V} 
S. Venkatesh. Representations of general linear groups in the Verlinde
category. {\em arXiv: 2203.03158}.

\bibitem 
[V2]{V2} 
S. Venkatesh. Harish-Chandra pairs in the Verlinde
category in positive characteristic. {\em International Mathematics Research Notices} (2023) no. 18, 15475–15536.

\bibitem
[Z]{Z} A.~Zimmermann. Representation theory. A homological algebra point of view. {\em Algebra and Applications, 19}. Springer, Cham, 2014. xx+707 pp.

\end{thebibliography}
\end{document}